\documentclass[12pt,leqno]{article}
\usepackage{amssymb,amsfonts,amsmath,amsthm,amscd,mathrsfs}
\usepackage{MnSymbol}
\usepackage{PSFigure,psfrag}
\usepackage{graphicx}
\usepackage{centernot}
\usepackage{stmaryrd}

\def\IE{{\mathbb E}}

\def\IP{{\mathbb P}}
\def\IR{{\mathbb R}}

\def\IT{{\mathbb T}}

\def\IZ{{\mathbb Z}}
\def\IV{{\mathbb V}}

\def\IX{{\mathbb X}}

\def\n{\noindent}

\def\dis{\displaystyle}

\def\fr{\mbox{\footnotesize $\dis\frac{1}{2}$}}

\def\ov{\overline}
\def\ve{\varepsilon}
\def\f{\footnotesize}
\def\r{\rightarrow}
\def\point{{\mbox{\large $.$}}}
\def\wh{\widehat}
\def\wt{\widetilde}

\def\Le{\wt{\cL}\,\!^\ve}
\def\Pe{\wt{\IP}\,\!^\ve}
\def\Pen{\wt{\IP}\,\!^{\ve_n}}
\def\Ee{\wt{\IE}\,\!^\ve}
\def\Een{\wt{\IE}\,\!^{\ve_n}}

\def\cD{{\cal D}}

\def\cL{{\cal L}}

\def\cJ{{\cal J}}

\def\cF{{\cal F}}

\def\cG{{\cal G}}

\dimendef\dimen=0

\newtheorem{theorem}{Theorem}[section]
\newtheorem{lemma}[theorem]{Lemma}
\newtheorem{corollary}[theorem]{Corollary}
\newtheorem{proposition}[theorem]{Proposition}
\newtheorem{remark}[theorem]{Remark}

\begin{document}

\noindent
~

\bigskip
\begin{center}
{\bf AN INVISCID LIMIT TO AN EFFECTIVE ENERGY-ENSTROPHY DIFFUSION PROCESS}
\end{center}

\begin{center}
Alain-Sol Sznitman$^1$ and Klaus Widmayer$^2$
\\[2ex]
Preliminary Draft
\end{center}

\begin{abstract}
In this article we consider a stationary $N$-dimensional Galerkin-Navier-Stokes type evolution with Brownian forcing and random stirring (of arbitrarily small strength that plays the role of a regularization). We show, as a ``proof of concept'', that the stationary diffusion in an open two-dimensional cone constructed in a companion article, stands as the inviscid limit of the laws of the ``enstrophy-energy'' process of the $N$-dimensional diffusion process considered here, this regardless of the strength of the stirring. With the help of 
the quantitative condensation bounds of the companion article, we infer quantitative inviscid condensation bounds, which for suitable forcings show an attrition of all but the lowest modes in the inviscid limit.

\end{abstract}

\vfill 
\hfill July 2026
\vfill
\noindent
{\footnotesize
-------------------------------- \\
$^1\,$Department of Mathematics, ETH Zurich, CH-8092 Zurich, Switzerland\\
$^2\,$Institute of Mathematics, University of Zurich, CH-8057 Zurich, Switzerland and Faculty of 

\vspace{-0.5ex}
\noindent
~~\!Mathematics, University of Vienna, A-1090 Vienna, Austria
}


\newpage
\thispagestyle{empty}
~

\newpage
\setcounter{page}{1}

\setcounter{section}{-1}
\section{Introduction}
While the existence and uniqueness theory for stationary distributions of the two-dimen\-sional incompressible Navier-Stokes equation with Brownian forcing on a two-dimensional square or thin torus is by now a fairly mature subject, see for instance \cite{HairMatt06}, \cite{HairMatt11}, \cite{KuksShir12}, the nature of the stationary measures, which typically are irreversible in the cases of interest, remains poorly understood. When an inviscid limit with a proper balancing of forces is considered, little information is known about the limit behavior, see \cite{GlatSverVico15}, \cite{Kuks08}, and Chapter 5 of \cite{KuksShir12}. These features persist when finite dimensional evolutions corresponding to Galerkin-Navier-Stokes are instead considered. For instance, the known lower bounds on the expectation of the energy are presently blind to the inviscid limit procedure, and the expected trend that under suitable forcings, in the inviscid limit, the stationary distributions condense on the low Fourier modes, is poorly understood, see \cite{BeckCoopLordSpil20}, \cite{BoucSimo09}, (4.2) in \cite{LinkHohmEckh20}, \cite{Naza11} on p.~127.

\medskip
Gaussian measure has been instrumental in several aspects of related research topics. It has been a building block for invariant measures corresponding to very specific forcings, and for the construction of generalized solutions, see \cite{DaprDebu02}, and \cite{HairRosa24} for more recent developments. Also, in a large dimension (and deviation) perspective, it has highlighted a condensation mechanism on low modes in the context of constrained enstrophy and energy regimes, see Section 3.1 of \cite{BoucCorv10}. In the companion article \cite{SzniWidm26a}, Gaussian measure plays a leading role as well. We use it as a tool to devise the drift and diffusion matrix of a diffusion process in a two-dimensional cone. We then establish well-posedness, uniqueness of the stationary measure, as well as certain remarkable condensation bounds for this stationary measure. In the present work, as a ``proof of concept'', to highlight the role of the above two-dimensional process, we show that it is the inviscid limit of the law of the ``enstrophy-energy'' process of a stationary $N$-dimensional Galerkin-Navier-Stokes type evolution with Brownian forcing and random stirring (of arbitrarily small strength, and that plays the role of a regularization). With the help of the condensation bound that we prove in \cite{SzniWidm26a}, we deduce an {\it inviscid condensation bound}, which for instance highlights a ``transfer to low modes'' in the case of suitable forcings.

\medskip
We will now discuss the results in more detail.

\medskip
We refer to Section 1 for precise assumptions. We consider $\IR^N$ with $N = 2n$, $n \ge 4$, as well as two quadratic forms $|x|^2 = \Sigma_\ell \, x_\ell^2$ and $|x|^2_{-1} = \Sigma_\ell \, x_\ell^2 / \lambda_\ell$, where $\lambda_{2i} = \lambda_{2i - 1}$ (denoted by $\mu_i$), $1 \le i \le n$, is increasing and $\mu_1 = 1$. The first quadratic form corresponds to (twice) the enstrophy and the second quadratic form to (twice) the enstrophy in the Galerkin perspective explained in Remark \ref{rem1.1}. We also have a Gaussian measure $\mu$ on $\IR^N$ under which the coordinate functions $x_\ell$, $1\le \ell \le N$, are independent centered Gaussian variables with variance equal to $a/2$, with $a > 0$.

\medskip
The main object of study is a certain stationary diffusion process on $\IR^N$, which depends on a small parameter $\ve > 0$. Its generator is given by
\begin{equation}\label{0.1}
\wt{\cL}\,\!^\ve = \wt{L} + \mbox{\f $\dis\frac{1}{\ve}$} \;(B + \kappa \, \cD), 
\end{equation}
where
\begin{equation}\label{0.2}
\wt{L} = \Sigma_\ell \, \lambda_\ell \,\Big(\mbox{\f $\dis\frac{a}{2}$} \; (1 + \delta_\ell)\, \partial^2_\ell - x_\ell \, \partial_\ell\Big)
\end{equation}
with $\delta_\ell \in (-1,0]$ and $\delta_{2i} = \delta_{2i - 1}$, for $1 \le i \le n$ (this last assumption can actually be dispensed with, but is kept to streamline the exposition, see below (\ref{1.9})),
\begin{equation}\label{0.3}
B = \Sigma_\ell \, b(x)_\ell \, \partial_\ell ,
\end{equation}
\begin{samepage}
where $b(x)_\ell$ is for each $\ell$ a quadratic form in the coordinates of $x$ (it should be thought of as a drift of Galerkin-Navier-Stokes type, see Remark \ref{rem1.1}) such that
\begin{equation}\label{0.4}
\mbox{${\rm div} \, b = 0$, and $\langle x, b(x)\rangle = 0$, $\langle x, b(x)\rangle_{-1} = 0$, for all $x$ in $\IR^N$},
\end{equation}
$0 < \kappa \le 1$ controls the strength of the stirring (it plays a regularization role) induced by
\begin{equation}\label{0.5}
\cD = \fr \; \Sigma^M_1 \, Z^2_m ,
\end{equation}
where $Z_m, 1 \le m \le M$, are certain specific vector fields, quadratic or linear in the coordinates of $x$, see (\ref{1.14}), (\ref{1.15}), such that for each $1 \le m \le M$,
\begin{equation}\label{0.6}
\mbox{${\rm div} \, Z_m = 0$, and $\langle x, Z_m(x) \rangle = 0, \langle x, Z_m(x)\rangle_{-1} = 0$, for all $x$ in $\IR^N$}
\end{equation}
(these $(Z_m)_{1 \le m \le M}$ are shown to satisfy a full rank property in Proposition \ref{propA.3}).

\medskip
Alternatively, the above stationary diffusion can be seen as the law of the stationary solution of the Stratonovich stochastic differential equation on $\IR^N$:
\begin{equation}\label{0.7}
d X^\ve_t = \big(- \Lambda X^\ve_t + \mbox{\f $\dis\frac{1}{\ve}$} \;b(X^\ve_t)\big) \, dt + \Sigma_\ell \big(\lambda_\ell a(1 + \delta_\ell)\big)^{\frac{1}{2}} e_\ell \ d \beta_\ell(t)  + \Big(\mbox{\f $\dis\frac{\kappa}{\ve}$}\Big)^{1/2} \Sigma_1^M Z_m (X^\ve_t) \circ d \wt{\beta}_m(t)
\end{equation}
where $(\Lambda x)_\ell = \lambda_\ell \, x_\ell$, for $x \in \IR^N, 1 \le \ell \le N, (e_\ell)_{1 \le \ell \le N}$ is  the canonical basis of $\IR^N, \beta_\ell, 1 \le \ell \le N, \wt{\beta}_m, 1 \le m \le M$, are independent Brownian motions.

\medskip
The martingale problem formulation attached to $\wt{\cL}\,\!^\ve$, see Proposition \ref{prop1.3}, will however be more convenient in view of the main weak convergence result that we wish to prove in this work (see Theorem \ref{theo5.1}). Let us sketch its content.

\medskip
When one lets $\wt{\cL}^{\,\!\ve}$ act on a function $g(x) = \psi (|x|^2, |x|^2_{-1})$, with $\psi (u,v)$ smooth on $\IR^2$, due to (\ref{0.4}), (\ref{0.6}), one obtains (with $\wt{L}$ as in (\ref{0.2})):
\begin{equation}\label{0.8}
\begin{split}
\Le \,g (x) & = \wt{L} \,g(x)
\\[-0.5ex]
& = 2 a \Big(\Sigma_\ell \, \lambda_\ell (1 + \delta_\ell) \, x^2_\ell \, \partial^2_u \, \psi + 2 \Sigma_\ell (1 + \delta_\ell) \, x_\ell^2 \, \partial^2_{u,v} \psi + \Sigma_\ell \; \mbox{\f $\dis\frac{1}{\lambda_\ell}$} \; (1 + \delta_\ell) \, x_\ell^2 \, \partial^2_v \psi\Big)
\\[-0.5ex]
& + \; (a \Sigma_\ell \, \lambda_\ell (1 + \delta_\ell) - 2 \Sigma_\ell \, \lambda_\ell \, x_\ell^2) \, \partial_u \psi + (a \Sigma_\ell (1 + \delta_\ell) - 2 \Sigma_\ell \, x_\ell^2) \, \partial_v \psi
\end{split}
\end{equation}
(where the partial derivations of $\psi$ are evaluated at $u = |x|^2, v = |x|^2_{-1}$).

\medskip
The effective diffusion process constructed in the companion article \cite{SzniWidm26a} is stationary, and lives in the interior $\stackrel{_\circ}{C}$ of the two-dimensional cone $C = \{(u,v) \in \IR^2$; $0 \le v \le u \le \lambda_N v\}$. It has a generator $\wt{A}$, see (\ref{2.10}), where in essence the terms $x^2_\ell$ in the second and third line of (\ref{0.8}) are replaced by functions $q_\ell$ (see also (\ref{2.9}) and Remark \ref{remA.2}) that, a bit informally, can be characterized through:
\begin{equation}\label{0.9}
\begin{split}
q_\ell (u,v)  = &\; \mbox{a good version of the conditional expectation}
\\
& \; E^\mu \big[x^2_\ell \, \big| \, |x|^2 = u, |x|^2_{-1} = v\big], \; 1 \le \ell \le N
\end{split}
\end{equation} 
(incidentally, these functions do not depend on the parameter $a > 0$, see below (\ref{2.8})).
\end{samepage}

\medskip
The above recipe to define the operator $\wt{A}$ turns out to produce an elliptic diffusion operator with Lipschitz coefficients in $\stackrel{_\circ}{C}$, for which the martingale problem is well posed, and for which there is a unique stationary distribution, see \cite{SzniWidm26a} and (\ref{2.11}), (\ref{2.12}), (\ref{2.15}), (\ref{2.16}). We denote by $\wt{P}$ the corresponding law of the $\stackrel{_\circ}{C}$-valued stationary diffusion process, see (\ref{2.15}). Also (and in a simpler fashion), for each $\ve > 0$, the martingale problem attached to $\cL^\ve$ is well-posed and has a unique stationary distribution, see Proposition \ref{prop1.3}. We denote by $\wt{\IP}^\ve$ the corresponding law of the $\IR^N$-valued stationary diffusion process, and by $(X_t)_{t \ge 0}$ the canonical $\IR^N$-valued process.

\medskip
A central ``proof of concept'' result of this work, in view of \cite{SzniWidm26a}, is Theorem \ref{theo5.1}. It shows that
\begin{equation}\label{0.10}
\mbox{as $\ve \r 0$, the laws of $(|X_t|^2, |X_t|^2_{-1})_{t \ge 0}$ under $\wt{\IP}^\ve$ converge weakly to $\wt{P}$}.
\end{equation}
This corresponds to an averaging result, see for instance Chapters 6 and 7 of \cite{FreiWent12}, \cite{Kuks13}, \cite{KuksPiat08}. A special difficulty here has to do with the fact that the ``fast variables'' evolve in the space $\IX_{u,v} = \{x \in \IR^N; |x|^2 = u, |x|^2_{-1} = v\}$, with $(u,v) \in C$, which moves with the ``slow variables'' $u = |X_t|^2$, $v = |X_t|^2_{-1}$, and the slow variables cross singular rays corresponding to $u = \lambda_\ell \, v$, $3 \le \ell \le N-3$ (see also Figure 1 in Appendix A). The estimates on the $L^2$-norm of the stationary density of \cite{BedrLiss21} can be adapted to the present set-up, and are helpful, see Proposition \ref{prop3.2}, and \cite{ElesHerzLiss26} for recent developments on this topic.

\medskip
It is important to point out that in Theorem \ref{theo5.1}, see (\ref{0.10}), the limit distribution $\wt{P}$ does not depend on $\kappa$ (the strength of the stirring that plays a regularization role). By usual arguments, see Remark \ref{rem5.2}, this shows that one can prove a similar weak convergence to the same $\wt{P}$ as in (\ref{0.10}), in the case where $\kappa$ in (\ref{0.1}) (or (\ref{0.9})) is replaced with a suitable $\kappa_\ve \r 0$, as $\ve \r 0$. It would be highly desirable to obtain a quantitative criterion on the decay to $0$ of $\kappa_\ve$ ensuring the convergence to $\wt{P}$, especially in the case of a drift $b(\cdot)$ stemming from the Galerkin projection of the non-linear term of the two-dimensional Navier-Stokes equation in vorticity form on a torus, as in (\ref{1.12}).

\medskip
Another notable result in the present article is that the weak convergence statement (\ref{0.10}) can be combined with the remarkable quantitative condensation bounds established in Theorem 5.1 of the companion article \cite{SzniWidm26a}, to yield in Theorem \ref{theo6.1} (which is more precise) that defining
\begin{equation}\label{0.11}
B_0 = a \, \Sigma_\ell (1 + \delta_\ell) < B_1 = a \, \Sigma_\ell \, \lambda_\ell (1 + \delta_\ell)
\end{equation}
(incidentally $2 \wt{\IE}^\ve [| X_t|^2] = B_0$ for all $\ve > 0$ and $t \ge 0$, see (\ref{3.4})), one has
\begin{equation}\label{0.12}
\begin{array}{l}
\mbox{for any $\ell_0 \in \{3, \dots, N\}$},
\\
2 \lim\limits_{\ve \r 0} \, \wt{\IE}^\ve \, [|X_0|^2 - |X_0|^2_{-1}] \le \mbox{\f $\dis\frac{B_1 - B_0}{\lambda_{\ell_0} - 1}$} + \mbox{\f $\dis\frac{\lambda_3}{\lambda_3- 1}$} \; \mbox{\f $\dis\frac{\ell_0}{N- \ell_0}$} \; B_0
\end{array}
\end{equation}
(the limit exists and the quantity under the expectation in the left member is non-negative by definition of the norms $| \cdot |, | \cdot |_{-1}$).

\medskip
Thus, when $\ell_0$ can be chosen so that $B_1 / B_0 \ll \lambda_{\ell_0}$ and $\ell_0 \ll N$, the left member of (\ref{0.12}) is small compared to $B_0 = 2 \wt{\IE}^\ve [|X_0|^2]$. This corresponds to an inviscid condensation, we refer to Remark \ref{rem6.2} for a more detailed discussion.

\medskip
Let us finally describe the organization of this article.

\medskip
Section 1 presents the set-up. Section 2 recalls the results from the companion article \cite{SzniWidm26a} concerning the functions $q_\ell$, the stationary diffusion in the cone $\stackrel{_\circ}{C}$, and the quantitative condensation bound, which it satisfies. Section 3 contains several a priori estimates. Section 4 proves the tightness of the laws of the ``enstrophy-energy'' process $(|X_t|^2, |X_t|^2)_{t \ge 0}$ under the $\wt{\IP}^\ve, \ve > 0$, in Theorem \ref{theo4.1}. Section 5 is mainly devoted to the proof of Theorem \ref{theo5.1} (see (\ref{0.10})). The short section 6 contains the inviscid condensation bound in Theorem \ref{theo6.1} ((\ref{0.12}) is a special case). Appendix A collects several results converging the spaces $\IX_{u,v}$, the conditional measure $\mu_{u,v}$, and the full rank property of the vector fields $Z_m, 1 \le m \le M$. Appendix B contains the main estimate on the convergence to equilibrium for the ``fast motion'' on the space $\IX_{u,v}$ when $u,v$ remain in a ``good set'' in Proposition \ref{propB.2}. Finally, the convention concerning positive constants is stated at the beginning of each section where it is needed. Numbered constants have a fixed value throughout the article.

\bigskip\bigskip\n
{\bf Acknowledgment:}  The authors wish to thank Tobias Rohner for extensive simulations of the incompressible Navier-Stokes equation with Brownian forcings on a two-dimensional torus.

\section{The set-up}
\setcounter{equation}{0}

In this section we introduce some notation as well as various objects, which we will encounter throughout this article. We have an even integer
\begin{equation}\label{1.1}
\mbox{$N = 2n$, with $n \ge 4$},
\end{equation}
as well as the positive $\lambda_\ell, 1 \le \ell \le N$, such that
\begin{equation}\label{1.2}
1 = \lambda_1 = \lambda_2 < \lambda_3 = \lambda_4 < \dots < \lambda_{2i-1} = \lambda_{2i} < \dots < \lambda_{N-1} = \lambda_N .
\end{equation}
We often use the notation
\begin{equation}\label{1.3}
\mbox{$\mu_i = \lambda_{2i}$, for $1 \le i \le n$, with $n$ as in (\ref{1.1})}.
\end{equation}
We introduce the scalar products on $\IR^N$ denoted by $\langle \cdot, \cdot \rangle$ and $\langle \cdot, \cdot\rangle_{-1}$ associated to the square norms:
\begin{equation}\label{1.4}
|x|^2 = \Sigma_\ell \, x_\ell^2 \ge |x|^2_{-1} = \Sigma_\ell \; \mbox{\f $\dis\frac{x_\ell^2}{\lambda_\ell}$} , \; \mbox{for $x \in \IR^N$}.
\end{equation}
We sometimes refer to the coordinates $x_\ell$ of $x$ as {\it modes}, $x_1$ and $x_2$ being the {\it lowest modes}, see (\ref{1.2}) and Remark \ref{rem1.1} below. Further, we have a Gaussian measure $\mu$ on $\IR^N$ corresponding to 
\begin{equation}\label{1.5}
a > 0 
\end{equation}
and
\begin{equation}\label{1.6}
d \mu = (\pi a)^{-N/2} \; \exp\Big\{- \mbox{\f $\dis\frac{|x|^2}{a}$}\Big\} \; dx
\end{equation}
(under $\mu$ the coordinates $x_\ell, 1 \le \ell \le N$, are i.i.d. centered Gaussian variables with variance equal to $\frac{a}{2}$).

Our main object of interest is a certain stationary diffusion process on $\IR^N$, see Proposition \ref{prop1.3} below. It depends on a (small) parameter
\begin{equation}\label{1.7}
\ve > 0
\end{equation}
and we will sometimes restrict $\ve$ to $(0,1]$.

\medskip
Among the ingredients to describe this diffusion are the coefficients
\begin{equation}\label{1.8}
\mbox{$\delta_\ell \in (-1,0]$, with $\delta_{2i} = \delta_{2i-1}$, for $1 \le i \le n$},
\end{equation}
which enter the definition of the diffusion operator
\begin{equation}\label{1.9}
\wt{L} = \Sigma_\ell \, \lambda_\ell \, \Big(\mbox{\f $\dis\frac{a}{2}$} \; (1 + \delta_\ell) \, \partial^2_\ell - x_\ell \, \partial_\ell\Big).
\end{equation}
Viewed as a generator, it corresponds to independent Ornstein-Uhlenbeck processes for each coordinate, with respective stationary variances $\frac{a}{2} \, (1 + \delta_\ell)$ and speed-up factor $\lambda_\ell$. The assumption that $\delta_{2i} = \delta_{2i - 1}$ in (\ref{1.8}), can actually be removed due to the fact that $q_{2i} = q_{2i-1}, \lambda_{2i} = \lambda_{2i-1}$, and one has the identity (\ref{A.21}), see also below (3.6) of \cite{SzniWidm26a}. But we keep it here because it somewhat simplifies the exposition.

\medskip
Further, there is a drift-type operator on $\IR^N$
\begin{equation}\label{1.10}
\mbox{$B = \Sigma_\ell \,b(x)_\ell \, \partial_\ell$, where $b(x)_\ell$ is for each a quadratic form in the coordinates of $x$},
\end{equation}
such that
\begin{equation}\label{1.11}
\mbox{${\rm div} \, b=0$, and $\langle x, b(x) \rangle = 0$, $\langle x, b(x) \rangle_{-1} = 0$  for all $x \in \IR^N$}.
\end{equation}

\begin{remark}\label{rem1.1} \rm The main motivating example for such a $b(\cdot)$ corresponds to the case of a suitable Galerkin projection for the non-linear term in the two-dimensional incompressible Navier-Stokes equation in vorticity form on a possibly thin torus $\IT = (\IR / 2 \pi \mu \IZ) \times (\IR / 2 \pi \IZ)$, with $0 < \mu \le 1$, when $\lambda_\ell, 1 \le \ell \le N$, are eigenvalues of the Laplacian on $\IT$, and $\varphi_\ell$ are $L^2(\IT)$-normalized, pairwise orthogonal eigenfunctions of the Laplacian attached to the $\lambda_\ell, 1 \le \ell \le N$, the space being chosen so that (\ref{1.1}), (\ref{1.2}) holds. Then, one sets
\begin{equation}\label{1.12}
\begin{array}{l}
b(x)_\ell = \dis\sum\limits_{1 \le a, b\le N} x_a \, x_b \, \Big(\mbox{\f $\dis\frac{1}{\lambda_a}$} - \mbox{\f $\dis\frac{1}{\lambda_b}$}\Big) \,  t_{a,b,\ell}, \; \mbox{for $1 \le \ell \le N$, with}
\\
\\[-1ex]
t_{a,b,\ell} = \fr \;\langle {\rm det} \, (\nabla \varphi_a, \nabla \varphi_b), \varphi_\ell \rangle_{L^2(\IT)}, \; \mbox{for $1 \le a, b, \ell \le N$}.
\end{array}
\end{equation}
The identities (\ref{1.11}) can straightforwardly be checked in this case, using the fact that for $a,b,c$ in $\{1,\dots,N\}$
\begin{equation}\label{1.13}
\mbox{$t_{a,b,c} = - t_{b,a,c}$ and $t_{a,b,c} = t_{c,a,b}$}
\end{equation}
(the first equality is straightforward, and the second equality follows by integration by parts in the formula defining $t_{a,b,c}$). Hence, $t_{a,b,c} \not= 0$ implies that $a,b,c$ are distinct. The first equality of (\ref{1.11}) is then immediate, and for the second and third it suffices to regroup in the corresponding sums the terms involving three given coordinates of $x$. \hfill $\square$
\end{remark}
The last ingredient to describe the diffusion on $\IR^N$ has to do with ``stirring''. For this purpose, we introduce vector fields on $\IR^N$ (which we also view as differential operators). 

\medskip
First, we let $\cJ$ stand for the collection of triplets $J = (k,\ell,m)$ in $\{1, \dots , N\}$ with $\lambda_k < \lambda_\ell < \lambda_m$ and for $J$ in $\cJ$ we set (there are 3 terms in the sum):
\begin{equation}\label{1.14}
T_J(x) = \dis\sum\limits_{\mbox{\f $(a,b,c)$ cyclic permutation of $J$}} x_a \,x_b \,  \Big(\mbox{\f $\dis\frac{1}{\lambda_a}$} - \mbox{\f $\dis\frac{1}{\lambda_b}$}\Big) \, \partial_c\,,
\end{equation}
(it has a bit the flavor of the constitutive elements of (\ref{1.12}), however without the presence of the $t_{a,b,c})$, and for $1 \le i \le n$, we set
\begin{equation}\label{1.15}
R_i(x) = x_{2i} \, \partial_{2 i - 1} - x_{2i-1} \, \partial_{2i}\,
\end{equation}
(it generates a rotation in the space spanned by $e_{2i-1}$ and $e_{2i}$). One then endows $\cJ$ with some order (for instance alphabetic) to label the elements of $\cJ$ as $J_m, 1 \le m \le | \cJ |$, and one defines the vector fields
\begin{equation}\label{1.16}
\begin{split}
Z_m = &\; T_{\cJ_m}, \; \mbox{for $1 \le m \le | \cJ |$},
\\
= &\; R_i, \; \mbox{for $m = | \cJ | + i$, with $| \cJ | + 1 \le m \le M \stackrel{\rm def}{=} | \cJ | + n$}\,.
\end{split}
\end{equation}
The vector fields $Z_m, 1 \le m \le M$ satisfy similarly as in (\ref{1.11}):
\begin{equation}\label{1.17}
\mbox{${\rm div} \, Z_m = 0$, and $\langle x, Z_m(x) \rangle = 0, \langle x, Z_m(x)\rangle_{-1} = 0$ for all $x \in \IR^N$}
\end{equation}
(and are either quadratic or linear in $x$).

\medskip
The following observation will be useful: if $Z$ stands for one of the vector fields $Z_m$ in (\ref{1.16}) or for $B$ in (\ref{1.10}), if $g(x) = \psi (|x|^2, |x|^2_{-1})$ for $x$ in $\IR^N$, with $\psi$ a $C^1$-function on $\IR^2$, and $f$ is a $C^1$-function on $\IR^N$, then one has
\begin{equation}\label{1.18}
Z (g\, f) = g\, Z\, f \;\; \mbox{(in particular $Z \, g = 0$)}.
\end{equation}
Responsible for the ``stirring'' of the diffusion (and playing the role of a regularization) will be the operator 
\begin{equation}\label{1.19}
\cD = \fr \; \Sigma^M_1 \, Z^2_m\,,
\end{equation}
and the strength of stirring will be governed by the parameter
\begin{equation}\label{1.20}
0 < \kappa \le 1
\end{equation}
(which can be thought as arbitrarily small).

\medskip
We can now introduce the diffusion operator on $\IR^N$, see (\ref{1.9}), (\ref{1.10}), (\ref{1.19})
\begin{equation}\label{1.21}
\Le = \wt{L} + \mbox{\f $\dis\frac{1}{\ve}$} \;(B + \kappa \cD), \; \mbox{for $\ve > 0$}.
\end{equation}
We have (see (\ref{1.9}), (\ref{1.21}) for notation)
\begin{lemma}\label{lem1.2}
Given $\psi(u,v)$ a $C^2$-function on $\IR^2$, setting $g(x) = \psi ( | x |^2, |x|^2_{-1})$, for $x$ in $\IR^N$, one has for $\ve > 0$, $x \in \IR^N$
\begin{equation}\label{1.22}
\begin{split}
\Le \,g(x) = & \; \wt{L} \,g(x)
\\
= & \;2a \,\Big(\Sigma_\ell \, \lambda_\ell (1 + \delta_\ell) \, x_\ell^2 \, \partial^2_u \psi + 2 \Sigma_\ell (1 + \delta_\ell) \, x_\ell^2 \, \partial^2_{u,v} \psi + \Sigma_\ell \; \mbox{\f $\dis\frac{1}{\lambda_\ell}$} \;(1 + \delta_\ell) \, x^2_\ell \, \partial^2_v \psi\Big)
\\
& + (a \, \Sigma_\ell \, \lambda_\ell (1 + \delta_\ell) - 2 \Sigma_\ell \, \lambda_\ell \, x_\ell^2) \, \partial_u \psi + (a\,\Sigma_\ell ( 1 + \delta_\ell) - 2 \Sigma_\ell \, x_\ell^2) \, \partial_v \psi
\end{split}
\end{equation}
(the partial derivative of $\psi$ are evaluated at $u = |x|^2, v = |x|^2_{-1}$).

\medskip
Further, if $g_0(x) = \exp \{ \frac{1}{2a} \;|x|^2_{-1}\}$, for $x$ in $\IR^N$, then for $\ve > 0$,
\begin{equation}\label{1.23}
\Le \,g_0(x) \le -(a N - |x|^2) \, / \, 2a \, e^{|x|^2_{-1} / 2a} \le \mbox{\f $\dis\frac{N}{2}$} \; e^{N/2}, \; \mbox{for $x$ in $\IR^N$}.
\end{equation}
\end{lemma}

\begin{proof}
We begin with (\ref{1.22}). By (\ref{1.18}) we see that $(B + \kappa \cD) \, g = 0$ and the first equality follows, and the second equality follows from (\ref{1.9}) and direct computation.

\medskip
As for (\ref{1.23}), we have by (\ref{1.22})
\begin{equation}\label{1.24}
\begin{split}
\Le \, g_0(x) & = \Big\{2a \, / \, (2a)^2 \, \Sigma_\ell \; \mbox{\f $\dis\frac{1}{\lambda_\ell}$} \; (1 + \delta_\ell) \, x^2_\ell + (a \, \Sigma_\ell (1 + \delta_\ell) - 2 |x|^2) \, / \, (2a)\Big\} \, e^{|x|^2_{-1} / 2a}
\\
&\!\!\! \stackrel{(\ref{1.8})}{\le} (a N - |x|^2) \, / \, (2a) \; e^{|x|^2_{-1} / 2a}
\end{split}
\end{equation}
which becomes negative when $|x|^2 > a\, N$, and at most $N/2 \; e^{N/2}$ otherwise, whence (\ref{1.23}). 
\end{proof}

We now have the ingredients to construct the diffusion process on $\IR^N$ attached to $\Le$ by means of a martingale problem, which is well posed, and induces a strong Markov process on $\IR^N$, which has a unique stationary distribution. With this in mind we consider the space $C(\IR_+,\IR^N)$ of continuous $\IR^N$-valued functions on $\IR_+$, endowed with its canonical $\sigma$-algebra $\cG$, its canonical filtration $(\cG_t)_{t \ge 0}$ and its canonical process $(X_t)_{t \ge 0}$.

\begin{proposition}\label{prop1.3}
Given $x$ in $\IR^N$, there is a unique probability $\Pe_x$ on $C(\IR_+, \IR^N)$ such that
\begin{equation}\label{1.25}
\begin{array}{l}
\mbox{$\Pe_x(X_0 = x) = 1$, and for any smooth compactly supported function $f$ on $\IR^N$},
\\
\mbox{$f(X_t) - \dis\int^t_0 \Le \,f(X_s) \, ds, t \ge 0$, is a $(\cG_t)_{t \ge 0}$-martingale}.
\end{array}
\end{equation}
The family $\Pe_x, x \in \IR^N$, constitutes a strong Markov process with state space $\IR^N$, and
\begin{equation}\label{1.26}
\mbox{there is a unique stationary probability $\wt{\mu}_\ve$ for this Markov process}.
\end{equation}
Further, one has
\begin{equation}\label{1.27}
\dis\int_{\IR^N} \exp\Big\{\mbox{\f $\dis\frac{1}{2a}$} \;|x|^2_{-1}\Big\} \; d \wt{\mu}_\ve(x) < \infty\,,
\end{equation}
and
\begin{equation}\label{1.28}
\mbox{$\wt{\mu}_\ve$ is absolutely continuous with respect to the Lebesgue measure on $\IR^N$}.
\end{equation}
\end{proposition}

We will see stronger versions of {\rm (\ref{1.27}), (\ref{1.28})} in Section 3.

\begin{proof}
The proof is analogous (but much simpler) than the proof of Theorem 3.2, Theorem 4.1 and Proposition 4.3 of the companion article \cite{SzniWidm26a}. With (\ref{1.23}), one establishes a Lyapunov-Foster condition (see (3.12) - (3.13) of \cite{SzniWidm26a}), which ensures with Theorem 2.1 on p.~524 of \cite{MeynTwee93} the well-posedness of the martingale problem. With the ellipticity of $\Le$ and the Lyapunov-Foster condition resulting from (\ref{1.23}), one finds with Theorem 4.2 on p.~529 of \cite{MeynTwee93} that (\ref{1.26}), (\ref{1.27}) hold. Analogous arguments as in Proposition 4.3 of \cite{SzniWidm26a} yield (\ref{1.28}).
\end{proof}

Given $\ve > 0$, we then write
\begin{equation}\label{1.29}
\Pe = \dis \int \wt{\mu}_\ve (dx) \,\Pe_x ,
\end{equation}
for the stationary law of the diffusion on $\IR^N$ associated to $\Le$, (a probability on $C(\IR_+,\IR^N)$), and denote by $\Ee$ the corresponding expectation. Of special interest to us will be the small $\ve$ behavior of the law under $\Pe$ of the stochastic process
\begin{equation}\label{1.30}
\mbox{$(W_t)_{t \ge 0} = \big((U_t, V_t)\big)_{t \ge 0}$, where $U_t = |X_t|^2, V_t = |X_t|^2_{-1}$, for $t \ge 0$}.
\end{equation}

\begin{remark}\label{rem1.4} \rm 1) The process $(X_t)_{t \ge 0}$ under $\Pe_x$ from Proposition \ref{prop1.3} has the same law as the solution of the Stratonovich stochastic differential equation on $\IR^N$:
\begin{equation}\label{1.31}
\left\{ \begin{split}
dz_t & = \Big(\!\!- \Lambda z_t + \mbox{\f $\dis\frac{1}{\ve}$} \,b(z_t)\Big) dt + \Sigma_\ell \big(\lambda_\ell \,a(1 + \delta_\ell)\big)^{\frac{1}{2}} e_\ell \ d \beta_\ell(t) + \Big(\mbox{\f $\dis\frac{\kappa}{\ve}$}\Big)^{\frac{1}{2}} \Sigma_1^M Z_m(z_t)   \circ d \wt{\beta}_m(t)
\\
z_0 & = x,
\end{split}\right.
\end{equation}
where $(\Lambda x)_\ell = \lambda_\ell \,x_\ell$, for $x \in \IR^N, 1 \le \ell \le N, (e_\ell)_{1 \le \ell \le N}$ is the canonical basis of $\IR^N$, and $\beta_\ell, 1 \le \ell \le N, \wt{\beta}_m, 1 \le m \le M$, are independent Brownian motions. However, in view of the weak convergence result, we aim at in Theorem \ref{theo5.1}, the martingale problem formulation will be more convenient for us.

\bigskip\n
2) Typically, the stationary measure $\wt{\mu}_\ve$ from Proposition \ref{prop1.3} is not explicit. However, when all $\delta_\ell$ in (\ref{1.8}) are equal, i.e.~when
\begin{equation}\label{1.32}
\mbox{for some $\delta \in (-1,0], \delta_\ell = \delta$, for all $1 \le \ell \le N$},
\end{equation}
then
\begin{equation}\label{1.33}
\begin{array}{l}
\mbox{$\wt{\mu}_\ve$ is the Gaussian measure on $\IR^N$ for which all $x_\ell$ are independent centered}
\\
\mbox{Gaussian variables with variance $\frac{a}{2}\;(1 + \delta)$ (i.e.~$\mu$ in (\ref{1.6}) when $\delta = 0$)}.
\end{array}
\end{equation}
We prove the case $\delta = 0$, the general case follows by replacing $a$ with $a(1 + \delta)$. The martingale problem in (\ref{1.25}) being well-posed, it suffices by the theorem on p.~2 and Section 4 of \cite{Eche82} to show that for every smooth compactly supported $f$ on $\IR^N$ one has
\begin{equation}\label{1.34}
\dis\int d \mu(x) \; \Le \, f(x) = 0.
\end{equation}
To this end, one notes that when $g,h$ are smooth functions on $\IR^N$ growing at most polynomially together with their derivatives, one has by (\ref{1.11}), (\ref{1.17})
\begin{align}
\dis\int B \, g \;\, h \, d \,\mu & = - \dis\int g\;\,B \, h \, d \, \mu \;\; \mbox{and} \label{1.35}
\\[1ex]
\dis\int \cD \, g\;\, h \, d\,\mu & =  \dis\int g\;\, \cD \, h \, d \, \mu\,. \label{1.36}
\end{align}
Further, if $L = \Sigma_\ell \, \lambda_\ell ( \frac{a}{2} \; \partial^2_\ell - x_\ell \,\partial_\ell)$ corresponds to $\wt{L}$ when all $\delta_\ell = 0$, one also has classically
\begin{equation}\label{1.37}
\hspace{-4ex}\dis\int L\, g\;\,h \, d\, \mu  = - \dis\int g\;\, L \, h \, d \, \mu \,.
\end{equation}
Thus, choosing $g = f$ and $h = 1$, one finds (\ref{1.34}) since $\Le = L + \frac{1}{\ve} \,(B +  \kappa \cD)$. The claim (\ref{1.33}) follows for $\delta = 0$, and hence for general $\delta \in (-1,0]$, as explained above. \hfill $\square$
\end{remark}

We conclude this section with an identity for the stationary distribution $\wt{\mu}_\ve$ (or equivalently for $\Pe$), $\ve > 0$, which reflects the special role of enstrophy and energy, as they are conserved by the drift and the stirring, see (\ref{1.11}), and (\ref{1.17}).

\begin{proposition}\label{prop1.5}
Given $\psi(u,v)$ a smooth function on $\IR^2$, which together with all its derivatives has at most polynomial growth. Then, setting $g(x) = \psi (|x|^2_1,|x|^2_{-1})$ for $x$ in $\IR^N$, for every $\ve > 0$, the function $\wt{L} g(x)$ is $\wt{\mu}_\ve$-integrable and
\begin{equation}\label{1.38}
\dis\int d\, \wt{\mu}_\ve(x) \; \wt{L} g   = 0
\end{equation}
(we refer to {\rm (\ref{1.22})} for an expression of $\wt{L} g$ in terms of $\psi$).
\end{proposition}

\begin{proof}
By Proposition \ref{prop1.3} and (\ref{1.29}), (\ref{1.22}), and stochastic calculus, we know that
\begin{equation}\label{1.39}
M^\psi_t = g(X_t) - g(X_0) - \dis\int^t_0 \wt{L} g(X_s) \,ds, t \ge 0,
\end{equation}
is a continuous local martingale, with bracket process
\begin{equation}\label{1.40}
\langle M^\psi \rangle_t = \dis\int^t_0 (\wt{L} \,g^2 - 2 g\,\wt{L} \,g)(X_s)\,ds, t \ge 0.
\end{equation}
By the exponential bound (\ref{1.27}) and stationarity,  $g(X_t), \wt{L} g(X_t)$, $\langle M^\psi\rangle_t$ are $\Pe$-integrable for every $t \ge 0$. Hence, by Doob's Inequality, see \cite{KaraShre88}, p.~14, $(M^\psi_t)_{t \ge 0}$ is a continuous square integrable martingale. It now follows that
\begin{equation}\label{1.41}
0 = \Ee [M_1^\psi] \stackrel{\mbox{\f stationarity}}{=} - \,\Ee [ \wt{L} g (X_0)] = - \dis\int_{\IR^N} d \wt{\mu}_\ve \, \wt{L} g,
\end{equation}
and this proves (\ref{1.38}). 
\end{proof}

\section{The limit diffusion process}
\setcounter{equation}{0}

In this section we recall some of the main features of the stationary diffusion process constructed in \cite{SzniWidm26a}. It will govern the $\ve \r 0$ limit behavior of the law of $(W_t)_{t \ge 0}$ under $\Pe$, see Theorem \ref{theo5.1}.

\medskip
The state space of the limit diffusion process is the interior $\stackrel{_\circ}{C}$ of the cone
\begin{equation}\label{2.1}
C = \{w = (u,v) \in \IR^2; \; 0 \le v \le u \le \lambda_N v\} \quad \mbox{(with $\lambda_N$ as in (\ref{1.2}))}.
\end{equation}
To describe the diffusion operator on $\stackrel{_\circ}{C}$, we first need to recall from \cite{SzniWidm26a} the functions $q_\ell$ on $C$ (they are ``good versions'' of the conditional expectations $E^\mu [x_\ell^2 \, \big| \, |x|^2 = u, |x|^2_{-1} = v]$, see also Remark \ref{remA.2}). Namely, see Section 2 of \cite{SzniWidm26a}, one has a collection of functions $q_\ell, 1 \le \ell \le N$, on $C$ such that
\begin{align}
&\mbox{each $q_\ell$ is non-negative, homogeneous of degree $1$, and Lipschitz on $C$}, \label{2.2}
\\[1ex]
&q_{2i} = q_{2i-1}, \; \mbox{for $1 \le i \le n$}, \label{2.3}
\\[1ex]
&\mbox{$ u = \Sigma_1^N \, q_\ell(u,v)$ and $v = \Sigma^N_1 \; \mbox{\f $\dis\frac{1}{\lambda_\ell}$} \; q_\ell (u,v)$, for all $(u,v) \in C$}, \label{2.4}
\\
&\mbox{each $q_\ell$ is positive on $\stackrel{_\circ}{C}$ and coincides on each open sector $\Delta_i = \{w = (u,v) \in C$;} \label{2.5}
\\
&0 < \mu_{i-1} \,v < u < \mu_i v\}, 1 < i \le n \; \mbox{(see (\ref{1.3}) for notation), with a rational function} \nonumber
\\
&\mbox{in $u,v$}. \nonumber
\end{align}
Further, the functions $q_\ell$ have the boundary values
\begin{align}
&\mbox{$q_\ell(u,v) = 0$ for $u = v$ or $u = \lambda_N v$, when $3 \le \ell \le N-2$, and $u \ge 0$}, \label{2.6}
\\[1ex]
&\mbox{$q_\ell(u,u) = \mbox{\f $\dis\frac{u}{2}$}$ and $q_\ell(u,u/\lambda_N) = 0$, when $\ell = 1$ or $2$, and $u \ge 0$}, \label{2.7}
\\[0.5ex]
&\mbox{$q_\ell(u,u) = 0$ and $q_\ell(u,u/\lambda_N) =  \mbox{\f $\dis\frac{u}{2}$}$, when $\ell = N-2$ or $N$, and $u \ge 0$}.\label{2.8}
\end{align}
The functions $q_\ell$ do not depend on the parameter $a > 0$ in (\ref{1.5}), see Proposition B.2 of \cite{SzniWidm26a}, and they are regular versions of the conditional expectations of $x^2_\ell$ under $\mu$, given $|x|^2 = u$ and $|x|^2_{-1} = v$. Namely, one has
\begin{equation}\label{2.9}
\mbox{$E^\mu \big[ x_\ell^2 \, \big| \, |x|^2, |x|^2_{-1}\big] = q_\ell ( |x|^2, |x|^2_{-1}), \mu$-a.s., for $1 \le \ell \le N$}
\end{equation}
(the left member denotes the conditional expectation of $x_\ell^2$ given $|x|^2, |x|^2_{-1}$ under $\mu$).

\medskip
We also refer to Remark \ref{remA.2} for a strenghtened version of this identity. One then defines the differential operator $\wt{A}$ on the open cone $\stackrel{_\circ}{C}$, formally obtained by replacing the terms $x^2_\ell$ by $q_\ell$ in (\ref{1.22}):
\begin{equation}\label{2.10}
\begin{split}
\wt{A} \psi = & \;2a\Big\{\Sigma_\ell \, \lambda_\ell (1 + \delta_\ell) \,q_\ell \, \partial^2_u \psi + 2 \Sigma_\ell (1 + \delta_\ell) \, q_\ell \, \partial^2_{u,v} \psi + \Sigma_\ell \;\mbox{\f $\dis\frac{1}{\lambda_\ell}$} \;(1 + \delta_\ell) \,q_\ell \, \partial^2_v \psi\Big\}
\\
& + \big(a \Sigma_\ell \,\lambda_\ell (1 + \delta_\ell) - 2 \Sigma_\ell \,\lambda_\ell \, q_\ell \big) \, \partial_u \psi + \big(a \Sigma_\ell \,(1 + \delta_\ell) - 2 \Sigma_\ell  \, q_\ell \big) \, \partial_v \psi.
\end{split}
\end{equation}
As established in Section 3 of \cite{SzniWidm26a}, $\wt{A}$ is an elliptic second order differential operator with Lipschitz coefficients in $\stackrel{_\circ}{C}$. Moreover, it gives rise to a diffusion process on $\stackrel{_\circ}{C}$, through a well-posed martingale problem. Namely, letting $C(\IR_+, \stackrel{_\circ}{C})$ stand for the space of continuous $\stackrel{_\circ}{C}$-valued functions on $\IR_+$, $\cF$ for its canonical $\sigma$-algebra, $(\cF_t)_{t \ge 0}$, for its canonical filtration, and $(W_t)_{t \ge 0}$ for its canonical process, one has
\begin{samepage}
\begin{equation}\label{2.11}
\begin{array}{l}
\mbox{$\wt{P}_w (W_0 = w) = 1$ and for any smooth compactly supported function $\psi$ on $\stackrel{_\circ}{C}$},
\\
\mbox{$\psi(W_t) - \dis\int^t_0 \wt{A} \psi(W_s)\, ds, t \ge 0$ is an $(\cF_t)_{t \ge 0}$-martingale under $\wt{P}_w$}.
\end{array}
\end{equation}
Further, see Section 4 of \cite{SzniWidm26a}, the family $\wt{P}_w, w \in \,\stackrel{_\circ}{C}$, constitutes a strong Markov process on the state space $\stackrel{_\circ}{C}$, and see Theorem 4.1 and Proposition 4.3 of \cite{SzniWidm26a},
\begin{equation}\label{2.12}
\begin{array}{l}
\mbox{there is a unique stationary distribution $\wt{\pi}$ on $\stackrel{_\circ}{C}$ for the above Markov process}
\\
\mbox{and it is absolutely continuous with respect to the Lebesgue measure on $\stackrel{_\circ}{C}$}.
\end{array}
\end{equation}
Moreover, if $\alpha > 0$ and $\beta > 0$ are such that
\begin{equation}\label{2.13}
\begin{array}{l}
2(2 \alpha + 1) \;\max\limits_\ell (\lambda_\ell - 1) (1+ \delta_\ell) < \Sigma_\ell (\lambda_\ell - 1) (1+ \delta_\ell) \;\; \mbox{and}
\\
2(2 \beta + 1) \;\max\limits_\ell (\lambda_N - \lambda_\ell) (1+ \delta_\ell) < \Sigma_\ell (\lambda_N- \lambda_\ell) (1+ \delta_\ell),
\end{array}
\end{equation}
then
\begin{equation}\label{2.14}
\dis\int_{\stackrel{_\circ}{C}}\, (u-v)^{-(\alpha +1)} + (\lambda_N v - u)^{-(\beta +1)} + e^{\frac{v}{2a}} d\, \wt{\pi}(w) < \infty.
\end{equation}
The law of the stationary diffusion process will play an important role in Section 5. One thus defines on $C(\IR_+,\stackrel{_\circ}{C})$ the probability
\begin{equation}\label{2.15}
\mbox{$\wt{P} = \dis\int \wt{\pi}(dw) \, \wt{P}_w$ (and one writes $\wt{E}$ for the corresponding expectation)}.
\end{equation}
By Proposition 4.2. of \cite{SzniWidm26a} one has the following characterization of $\wt{P}$:
\begin{equation}\label{2.16}
\begin{array}{l}
\mbox{$\wt{P}$ is the only probability $Q$ on $C(\IR_+, \stackrel{_\circ}{C})$ such that $Q$ is stationary and}
\\[0.5ex]
\mbox{for any smooth compactly supported function $\psi$ on $\stackrel{_\circ}{C}$, under $Q$},
\\
\mbox{$\psi(W_t) - \psi(W_0) - \dis\int^t_0 \wt{A} \psi(W_s) \, ds, t \ge 0$, is an $(\cF_t)_{t \ge 0}$-martingale}.
\end{array}
\end{equation}
Finally, we recall the condensation bound. One first defines
\begin{equation}\label{2.17}
B_0 = a\, \Sigma_\ell (1 + \delta_\ell) \; \mbox{and} \; B_1 = a\, \Sigma_\ell \, \lambda_\ell (1 + \delta_\ell),
\end{equation}
so that $B_0 < B_1$. One knows from Proposition 4.4. of \cite{SzniWidm26a} that
\begin{equation}\label{2.18}
B_0 = 2 \wt{E} [U_0] \; \mbox{and} \; B_1 = 2 \wt{E} [\Sigma_\ell \, \lambda_\ell \, q_\ell(W_0)] .
\end{equation}
The {\it condensation bound}, (see Theorem 5.1 of \cite{SzniWidm26a}), states that 
\begin{equation}\label{2.19}
\begin{array}{l}
\mbox{for any $\ell_0$ in $\{3, \dots, N\}$, writing $\ell_0 = 2i_0$ or $\ell_0 = 2i_0 - 1$, with $2 \le i_0 \le n$,}
\\[0.5ex]
\mbox{(depending on whether $\ell_0$ is even or odd), one has}
\\[0.5ex]
2 \wt{E} [U_0 - V_0] \le \mbox{\f $\dis\frac{B_1 - B_0}{\lambda_{\ell_0} - 1}$} + \mbox{\f $\dis\frac{\lambda_3}{\lambda_3 - 1}$} \;\dis\sum\limits_{1 \le k < i_0 - 1} \; \mbox{\f $\dis\frac{1}{n-k}$} \;B_0 \le  \mbox{\f $\dis\frac{B_1 - B_0}{\lambda_{\ell_0} - 1}$} + \mbox{\f $\dis\frac{\lambda_3}{\lambda_3 - 1}$}  \; \mbox{\f $\dis\frac{\ell_0}{N - \ell_0}$} \;B_0
\end{array}
\end{equation}
(the sum in the middle member is omitted when $i_0 = 2$).
\end{samepage}

The proof of (\ref{2.19}) relies on the important {\it monotonicity property} of the $q_\ell, 3 \le \ell \le N$, see Theorem 2.3 of \cite{SzniWidm26a}. The condensation bound (\ref{2.19}) shows in particular that $\wt{E} [U_0 - V_0]$ is small compared to $\wt{E}[U_0] ( = B_0 / 2)$ when $\ell_0$ can be chosen so that $B_1 / B_0 \ll \lambda_{\ell_0}$ and $\ell_0 \ll N$. The ratio $B_1/B_0$ can be viewed as some ``effective spectral value'' at which the Brownian forcing induced by the $(\beta_\ell)_{1 \le \ell \le N}$ occurs, see (\ref{1.31}) and Remark \ref{rem6.2} 1).

\section{Some controls}
\setcounter{equation}{0}

In this section we collect several estimates on the stationary distribution $\wt{\mu}_\ve$, see (\ref{1.26}), which hold uniformly for $\ve \in (0,1)$ (or even for $\ve > 0$ for some of them). We also quote at the end of the section the basic estimate for the convergence to equilibrium of the diffusion under the sole effect of drift and stirring from Appendix B. Throughout this section, positive constants will a priori depend on $N, (\lambda_\ell)_{1 \le \ell \le N}, a, (\delta_\ell)_{1 \le \ell \le N}$, on the coefficients entering the definition of $b(\cdot)$ in (\ref{1.10}), and on $\kappa$, see (\ref{1.1}), (\ref{1.2}), (\ref{1.5}), (\ref{1.8}), (\ref{1.20}). Dependence on additional parameters will appear in the notation.

\medskip
We begin with some exponential controls on $\wt{\mu}_\ve$, which substantially sharpen (\ref{1.27}). We need some notation. We define the stochastic process on $C(\IR_+,\IR^N)$, see (\ref{1.30}) for notation:
\begin{equation}\label{3.1}
T_t = \Sigma_\ell \, \lambda_\ell \, X^2_{t,\ell} \; (\ge U_t \ge V_t), \; t \ge 0,
\end{equation}
and the positive reals:
\begin{align}
B_0 & = a \, \Sigma_\ell (1 + \delta_\ell), \; B^\prime_0 = \sup\limits_\ell \;a(1 + \delta_\ell) \; (\le B_0), \label{3.2}
\\[1ex]
B_1 & = a \, \Sigma_\ell \, \lambda_\ell (1 + \delta_\ell), \; B^\prime_1 = \sup\limits_\ell \;a\, \lambda_\ell (1 + \delta_\ell) \; (\le B_1). \label{3.3}
\end{align}
They measure the strength of the Brownian forcing attached to the $\beta_\ell(\cdot), 1 \le \ell \le N$, in (\ref{1.31}). The next proposition contains exponential controls on $U_0, V_0$ under $\Pe$, uniformly in $\ve > 0$, which sharpen (\ref{1.27}). These controls will be used repeatedly in the next sections.

\begin{proposition}\label{prop3.1}
For $\ve > 0$ one has
\begin{equation}\label{3.4}
2 \Ee [U_0] = B_0 \; \mbox{and} \; 2 \Ee [T_0] = B_1.
\end{equation}
Further, there is a finite non-negative function $\Phi$ on $\{(z,b,b^\prime) \in \IR^3_+; z b^\prime < 1\}$ such that
\begin{align}
\Ee [\exp\{z V_0\}(1 + U_0)] &\le \Phi (z,B_0,B^\prime_0), \; \mbox{for all $\ve > 0$ and  $0 \le z < 1/B^\prime_0$}, \label{3.5}
\\[1ex]
\Ee [\exp\{z U_0\}(1 + T_0)] &\le \Phi (z,B_1,B^\prime_1), \; \mbox{for all $\ve > 0$ and  $0 \le z < 1/B^\prime_1$}. \label{3.6}
\end{align}
\end{proposition}

\begin{proof}
The identities (\ref{3.4}) are immediate consequences of (\ref{1.38}) and (\ref{1.22}) with the choices $\psi(u,v) = v$ and $\psi(u,v) = u$.

\begin{samepage}
\medskip
We proceed with the proof of (\ref{3.6}). The choice of $\psi(u,v) = u^m$ with $m \ge 1$ in (\ref{1.38}) and (\ref{1.22}) yields that for arbitrary $\ve > 0$
\begin{equation}\label{3.7}
2a(m-1) \, \Ee [\Sigma_\ell \, \lambda_\ell (1 + \delta_\ell) \, X^2_{0,\ell} \,U_0^{m-2}] + a \, \Sigma_\ell \, \lambda_\ell (1 + \delta_\ell) \, \Ee_0 [U_0^{m-1}] = 2 \Ee [T_0 \,U_0^{m-1}]
\end{equation}
(where the first term is omitted when $m = 1$).
\end{samepage}

\medskip
Noting that $B_1 = 2 \Ee[T_0] \ge 2 \Ee[U_0]$, see (\ref{3.1}), and that by (\ref{3.3}) $a \, \Sigma_\ell \, \lambda_\ell (1 + \delta_\ell) \, X^2_{0,\ell} \le B^\prime_1 U_0$, one finds that for $m \ge 2$ and $\ve > 0$
\begin{equation}\label{3.8}
\Ee [U_0^m] \stackrel{(\ref{3.1})}{\le} \Ee [T_0 \,U_0^{m-1}] \stackrel{(\ref{3.7})}{\le} \Big\{(m-1) \,B^\prime_1 + \fr \;B_1\Big\} \; \Ee [U_0^{m-1}].
\end{equation}
iterating, and using $\Ee[U_0] \le B_1/2$, we find that
\begin{equation}\label{3.9}
\Ee [U_0^m] \le \prod\limits_{0 \le p < m}  \Big(p \,B^\prime_1 + \fr \;B_1\Big), \; \mbox{for $m \ge 1$ and $\ve > 0$},
\end{equation}
and using the second inequality of (\ref{3.8}) with $m$ in place of $m-1$, we find
\begin{equation}\label{3.10}
\Ee [T_0 \,U_0^m] \ \le \prod\limits_{0 \le p \le m}  \Big(p \,B^\prime_1 + \fr \;B_1\Big), \; \mbox{for $m \ge 1$ and $\ve > 0$}.
\end{equation}
We then note that for $z \ge 0$, one has
\begin{equation}\label{3.11}
\begin{array}{l}
\Ee [\exp \{z U_0\}(1 + T_0)] = \Sigma_{m \ge 0} \; \mbox{\f $\dis\frac{z^m}{m!}$} \; \Ee[U_0^m (1 + T_0)] \stackrel{(\ref{3.9}),(\ref{3.10})}{=}
\\[1ex]
\Sigma_{m \ge 0} \;  \mbox{\f $\dis\frac{z^m}{m!}$} \Big\{\prod\limits_{0 \le p < m} \Big(p \,B^\prime_1 + \fr \;B_1\Big)\Big\} \Big(1 + m \, B^\prime + \fr \; B_1\Big).
\end{array}
\end{equation}
Setting for $z,b,b^\prime \ge 0$ with $z b^\prime < 1$
\begin{equation}\label{3.12}
\Phi (z,b,b^\prime) = \Sigma_{m \ge 0} \;  \mbox{\f $\dis\frac{z^m}{m!}$} \;\Big\{\prod\limits_{0 \le p < m} \Big(p b^\prime + \fr \;b\Big)\Big\} \Big(1 + m b^\prime + \fr \;b\Big),
\end{equation}
which is finite by the ratio test, one obtains (\ref{3.6}).

\medskip
The proof of (\ref{3.5}) is similar. Using $\psi(u,v) = v^m$ with $m \ge 1$ in (\ref{1.38}) and (\ref{1.22}), one finds in place of (\ref{3.7}) that for $\ve > 0$,
\begin{equation}\label{3.13}
2a (m-1) \, \Ee\Big[ \Sigma_\ell \; \mbox{\f $\dis\frac{1}{\lambda_\ell}$}\; (1 + \delta_\ell)\, X^2_{0,\ell} \, V_0^{m-2}\Big] + a \Sigma_\ell (1 + \delta_\ell) \, \Ee [V_0^{m-1}] = 2 \Ee[U_0 V_0^{m-1}]
\end{equation}
(where the first term is omitted when $m=1$).

\medskip
One now uses $B_0 = 2 \Ee[U_0] \ge 2 \Ee [V_0], \; \mbox{and} \; a \Sigma_\ell \, \frac{1}{\lambda_\ell} \, (1 + \delta_\ell)\, X^2_{0,\ell} \le B^\prime_0 V_0$, to deduce that for $m \ge 2$ and $\ve > 0$,
\begin{equation}\label{3.14}
\Ee [V_0^m] \le \Ee [U_0 V_0^{m-1}] \stackrel{(\ref{3.13})}{\le} \Big\{ (m-1) \,B^\prime_0 + \fr \; B_0\Big\} \; \Ee[V_0^{m-1}].
\end{equation}
The proof then runs similar to that of (\ref{3.6}) and yields (\ref{3.5}). 
\end{proof}

As an aside, similar bounds could be derived in the context of stationary incompressible Navier-Stokes equations with Brownian forcing on a two-dimensional torus and would sharpen the controls on p.~212 of \cite{KuksShir12}.

\medskip
We now proceed with a result which substantially sharpens (\ref{1.28}), where it is stated that the stationary distribution $\wt{\mu}_\ve$ is absolutely continuous with respect to the Lebesgue measure. Its proof is due to \cite{BedrLiss21}. It will play an important role in showing that the near ``singular set'' where $|x|^2 / |x|^2_{-1}$ gets close to some of the $\lambda_\ell$ (see also Appendix B) has uniformly small measure under $\wt{\mu}_\ve$, $0 < \ve \le 1$.

\begin{proposition}\label{prop3.2}
Letting $\wt{h}_\ve$ stand for the density of $\wt{\mu}_\ve$ with respect to the Lebesgue measure on $\IR^N$, one has
\begin{equation}\label{3.15}
\sup\limits_{0 < \ve \le 1} \; \dis\int \wt{h}_\ve^2 \, dx \le C
\end{equation}
(and of course $\wt{h}_\ve \ge 0$ with $\int h_\ve\, dx = 1$).
\end{proposition}

\begin{proof}
The proof is the adaptation to our simpler context of the proof of (4.4) on pp.~507-508 of \cite{BedrLiss21}, noting that we are here in an elliptic set-up due to (\ref{1.8}), and the stirring vector fields present here, see (\ref{1.16}) - (\ref{1.17}), actually improve the energy estimate corresponding to (4.6) of \cite{BedrLiss21}. Incidentally, the notation of \cite{BedrLiss21} is matched with our set-up by comparing (1.3) of  \cite{BedrLiss21} with (\ref{1.31}) here. In the notation of  \cite{BedrLiss21} we are interested in the case $\alpha = 0, A = 0$, $B = \Lambda, N = -b(\cdot)$, and the vectors $Z_j, 1 \le j \le r$ of  \cite{BedrLiss21} correspond to $(\lambda_\ell \,a(1 + \delta_\ell))^{1/2} e_\ell, 1 \le \ell \le N$, here. In the present article, we also have the stirring vector fields $Z_m, 1 \le m \le M$, which improve the estimates needed in the proof of (\ref{3.15}) when adapting the proof of (4.4) in  \cite{BedrLiss21}.
\end{proof}

The next proposition will be especially useful in Section 5, in particular when using Proposition \ref{propB.2}, see also (\ref{5.24}), (\ref{5.47}) to benefit from a rapid convergence to equilibrium on $\IX_{u,v}$, see (\ref{A.1}), under the sole action of drift and stirring. Given $0 < u_{\min} < u_{\max}$ and $\eta > 0$, we introduce the ``untamed subset'' of $\IR^N$
\begin{equation}\label{3.16}
\begin{array}{l}
U(u_{\min}, u_{\max}, \eta) =
\\
\big\{x \in \IR^N; u_{\min} \le |x|^2 \le u_{\max}, \;\mbox{and $\big| \,|x|^2 / |x|^2_{-1} - \lambda_\ell  \big| \ge \eta$,  for all $1 \le \ell \le N\big\}^c$},
\end{array}
\end{equation}
as well as
\begin{equation}\label{3.17}
\gamma(u_{\min}, u_{\max}, \eta) = \sup\limits_{0 < \ve \le 1} \;\wt{\mu}_\ve \big(U (u_{\min}, u_{\max}, \eta)\big).
\end{equation}

\begin{proposition}\label{prop3.3}
\begin{equation}\label{3.18}
\lim\limits_{u_{\min} \r 0, u_{\max} \r \infty, \eta \r 0} \gamma (u_{\min}, u_{\max}, \eta) = 0.
\end{equation}
\end{proposition}

\medskip
\begin{proof}
Note that $B_0$ and $B_1$ are constants in view of the convention stated at the beginning of this section. Thus, by (\ref{3.4}) and the Chebychev Inequality, we see that
\begin{equation}\label{3.19}
\lim\limits_{E \r \infty} \; \sup\limits_{\ve > 0} \; \wt{\mu}_\ve (|x|^2 \ge E) = 0.
\end{equation}
Moreover, by the Cauchy-Schwarz Inequality and Proposition \ref{prop3.2}, we find that for $\ve \in (0,1]$ and $e > 0$ one has
\begin{equation}\label{3.20}
\wt{\mu}_\ve (|x|^2 \le e) = \dis\int_{|x|^2 \le e} \wt{h}_\ve\, dx \le \Big(\dis\int \wt{h}^2_\ve\, dx\Big)^{1/2} | \{x; |x|^2 \le e\}|^{1/2} \le C \,|\{x; |x|^2 \le e\}|^{1/2},
\end{equation}
and hence
\begin{equation}\label{3.21}
\lim\limits_{e \r 0} \; \sup\limits_{0 < \ve \le 1} \wt{\mu}_\ve ( | x|^2 \le e) = 0.
\end{equation}
Then, with the help of (\ref{3.19}), (\ref{3.20}) and a similar argument as above
\begin{equation}\label{3.22}
\lim\limits_{\eta \r 0} \; \sup\limits_{0 < \ve \le 1} \;\wt{\mu}_\ve (\{x \not= 0; \;\mbox{for some} \; 1 \le \ell \le N, |x|^2 / |x|^2_{-1} \in [\lambda_\ell - \eta, \lambda_\ell + \eta]\}) = 0.
\end{equation}
Together with (\ref{3.19}), (\ref{3.20}) this proves (\ref{3.18}).
\end{proof}

\medskip
We finally recall the main control concerning the effect of the sole action of {\it drift and stirring} from Appendix B. Given $u_{\min}, u_{\max}$ and $\eta$ as above (\ref{3.16}), we introduce the good set
\begin{equation}\label{3.23}
\begin{split}
G(u_{\min}, u_{\max}, \eta)  = &\;U (u_{\min}, u_{\max}, \eta)^c
\\
 = &\; \big\{x \in \IR^N; u_{\min} \le |x|^2 \le u_{\max}, \big| \, |x|^2 / |x|^2_{-1} - \lambda_\ell \, \big| \, \ge \eta,
\\
&\quad  \mbox{for all $1 \le \ell \le N\big\}$}.
\end{split}
\end{equation}
Then, given $x$ in $G(u_{\min}, u_{\max}, \eta)$, setting $u = |x|^2, v = |x|^2_{-1}$, we recall the space $\IX_{u,v}$ from (\ref{A.1}), and the regular conditional probability $\mu_{u,v}$ of $\mu$ given $| \cdot |^2 = u$ and $| \cdot |^2_{-1} = v$, see (\ref{A.7}), (\ref{A.12}). The diffusion with generator $\Gamma = \kappa \, \cD + B$ (see (\ref{1.10}), (\ref{1.19})) preserves $\IX_{u,v}$ and has a transition density $p_t(y,z), t > 0$, with respect to $\mu_{u,v}(dz)$ on $\IX_{u,v}$. By Proposition \ref{B.2}, one has for all $y,z$ in $\IX_{u,v}$ and $t \ge 1$,
\begin{equation}\label{3.24}
|p_t (y,z) - 1| \le c_2(u_{\min}, u_{\max},\eta) \; \exp\{-c_3(u_{\min}, u_{\max},\eta) \;t\}.
\end{equation}
This estimate will be especially useful in the proof of Theorem \ref{theo5.1}.

\section{Tightness}
\setcounter{equation}{0}

In this short section we prove the tightness of the laws of the two-dimensional process $(W_t)_{t \ge 0}$ in (\ref{1.30}) under the stationary laws $\Pe, \ve > 0$, see (\ref{1.29}). In essence, this process corresponds to the time evolution of the ``enstrophy and energy'' for the stationary diffusions with generator $\Le$, see (\ref{1.21}). We keep the same convention concerning constants as stated at the beginning of Section 3.

\medskip
We recall (\ref{2.1}) for the definition of the closed cone $C$ in $\IR^2$. We endow the space $C(\IR_+, C)$ of continuous $C$-valued functions on $\IR_+$ with the topology of uniform convergence on compact intervals, and with its usual Borel $\sigma$-algebra, which coincides with the $\sigma$-algebra generated by the canonical coordinates (see for instance Chapter 16 of \cite{Kall02}). We denote by
\begin{equation}\label{4.1}
\mbox{$\wt{P}^\ve$ the law of $(W_t)_{t \ge 0}$ on $C(\IR_+,C)$ under $\Pe$, for $\ve > 0$}.
\end{equation}
The main object of this section is the proof of
\begin{theorem}\label{theo4.1}
\begin{equation}\label{4.2}
\mbox{The laws $\wt{P}^\ve, \ve > 0$, are tight}.
\end{equation}
\end{theorem}

\begin{proof}
We first note that $C(\IR_+,C)$ sits as a closed subspace of $C(\IR_+, \IR^2)$ endowed with the topology of uniform convergence on compact intervals, 
and it thus suffices to show the tightness of the laws of $(W_t)_{t \ge 0} = ((U_t, V_t))_{t \ge 0}$ as a continuous $\IR^2$-valued process under $\Pe, \ve > 0$. The first identity in (\ref{3.4}) and the inequality $U_0 \ge V_0 \ge 0$ show that
\begin{equation}\label{4.3}
\mbox{the laws of $W_0 = (U_0,V_0)$ on $\IR^2$ under $\Pe, \ve > 0$, are tight}.
\end{equation}
To complete the proof of (\ref{4.2}) we will show that
\begin{align}
& \sup\limits_{\ve > 0, s \ge 0} \; \Ee[ | U_{t + s} - U_s|^4] \le c \, t^2, \; \mbox{for $0 \le t \le 1$}, \label {4.4}
\\[1ex]
& \sup\limits_{\ve > 0, s \ge 0} \; \Ee[ | U_{t + s} - V_s|^4] \le c \, t^2, \; \mbox{for $0 \le t \le 1$}, \label {4.5}
\end{align}
and the claim will come from the application of the Kolmogorov-Chentsov Theorem, see \cite{Kall02}, p.~313.

\medskip
We first prove (\ref{4.4}). By stationarity we can assume that $s = 0$. We apply (\ref{1.39}), (\ref{1.40}) (see also below (\ref{1.40})), to the case $\psi(u,v) = u$ and find that with the notation (\ref{3.1}), (\ref{3.3}), under $\Pe, \ve > 0$,
\begin{equation}\label{4.6}
\mbox{$M_t = U_t - \dis\int^t_0 B_1 - 2T_s \, ds, t \ge 0$, is a continuous $(\cG_t)_{t \ge 0}$-martingale},
\end{equation}
with bracket process
\begin{equation}\label{4.7}
\langle M \rangle_t = 4 a  \dis\int^t_0 \Sigma_\ell \, \lambda_\ell (1 + \delta_\ell) \, X^2_{s,\ell} \, ds \stackrel{(\ref{1.8})}{\le} 4 a \dis\int^t_0 T_s \, ds, \; \mbox{for $t \ge 0$}.
\end{equation}
We then find that for $t > 0$
\begin{equation}\label{4.8}
\begin{split}
(U_t - U_0)^4 & \le c \,\Big(\dis\int^t_0 B_1 - 2 T_s \, ds\Big)^4 + c (M_t - M_0)^4
\\
& \le c\, t^3 \dis\int^t_0 (B_1 - 2 T_s)^4 \, ds + c (M_t - M_0)^4 .
\end{split}
\end{equation}
By the Burkholder-Davis-Gundy Inequalities, see \cite{Kall02}, p.~333, one has
\begin{equation}\label{4.9}
\Ee [ \sup\limits_{0 \le s \le t} (M_s - M_0)^4] \le c \, \Ee [\langle M\rangle^2_t] \stackrel{(\ref{4.7})}{\le} c \, t^{2} \, \Ee[T_0^2] .
\end{equation}
Hence, we find that for $0 \le  t \le 1$,
\begin{equation}\label{4.10}
\Ee [ \sup\limits_{0 \le s \le t} (U_s - U_0)^4] \le c \,t^4 \, \Ee [(1 + T_0)^4] + c t^2 \,\Ee [T_0^2] \le c\, t^2,
\end{equation}
using $T_0 \le c \,U_0$ (recall that $\lambda_N$ is a constant by the convention stated at the beginning of Section 3) and (\ref{3.6}) in the last step.

\medskip
The inequality (\ref{4.5}) is proved in a similar (even simpler) fashion. This completes the proof of Theorem \ref{theo4.1}.
\end{proof}

\section{The weak convergence result}
\setcounter{equation}{0}

In this section we prove a central result of this article, namely the weak convergence of the $\wt{P}^\ve$, $\ve > 0$, see (\ref{4.1}), to the law $\wt{P}$ in (\ref{2.15}). This provides a meaning to the interpretation of $\wt{P}$ as the law of an {\it effective enstrophy-energy diffusion process}, and offers a rational for the averaging procedure consisting in the replacement of $x^2_\ell$ in (\ref{1.22}) by $q_\ell$ in (\ref{2.10}). We keep the same convention as stated at the beginning of Section 3 concerning constants.

\medskip
Before stating our main result, we begin with some observations. We view $C(\IR_+,\stackrel{_\circ}{C})$ as a Borel subset of $C(\IR_+,C)$ (made of trajectories with infinite entrance time in $\stackrel{_\circ}{C}\,\!\!\!^c$), and thus we view $\wt{P}$ in (\ref{2.15}) as a probability on $C(\IR_+,C)$ giving full mass to $C(\IR_+,\stackrel{_\circ}{C})$. The laws of $(\wt{P}_\ve)_{\ve > 0}$ in (\ref{4.1}) are also defined on $C(\IR_+,C)$. The main result of this section is

\begin{theorem}\label{theo5.1}
\begin{equation}\label{5.1}
\mbox{As $\ve \r 0$, the laws $\wt{P}^\ve$converge weakly to $\wt{P}$}.
\end{equation}
\end{theorem}

Before turning to the proof of Theorem \ref{theo5.1} we make the important observation:

\begin{remark}\label{rem5.2} \rm
The limit law $\wt{P}$ in (\ref{2.15}) does not depend on the parameter $\kappa$ in (\ref{1.20}). The topology of weak convergence on $C(\IR_+,C)$ is metrizable, see for instance \cite{StroVara79}, p.~9, and it follows that a convergence to the same $\wt{P}$ holds when the fixed parameter $\kappa \in (0,1]$, measuring the strength of the stirring and entering the definition of $\Pe$ in (\ref{1.29}) and Proposition \ref{prop1.3}, is replaced by a suitable $\kappa_\ve \in (0,1]$ tending to zero as $\ve$ goes to zero. Of course, it would be desirable to have a quantitative control on the rate of decay of $\kappa_\ve$ to $0$ ensuring the weak convergence to $\wt{P}^\ve$ for the law of $(W_t)_{t \ge 0}$ under the stationary solution of the martingale problem attached to $\wt{L} + \frac{1}{\ve} \,(B + \kappa_\ve \cD)$, as $\ve$ goes to zero. \hfill $\square$
\end{remark}

Theorem \ref{theo5.1} can be viewed as an instance of an averaging result for the fast variables, see Chapter 6 and 7 of \cite{FreiWent12}. While the proof shares a similar spirit with \cite{Kuks13}, \cite{KuksPiat08}, a special feature here is that what plays the roles of ``fast variables'' evolves in the space $\IX_{u,v}$, which moves along the ``slow variables'' $u = U_t, v = V_t, t \ge 0$, and the slow variables may cross some singular rays corresponding to $u = \mu_i v$, with $2 \le i \le n-1$. The controls developed in Section 3 are helpful on this matter. The proof first proceeds with several reduction steps that bring to the core statement (\ref{5.10}) to be established.

\bigskip\n
{\it Proof of} Theorem \ref{theo5.1}: With the tightness shown in Theorem \ref{theo4.1}, we can thus consider an arbitrary sequence $\ve_n \r 0$, such that $\wt{P}^{\ve_n}$ converges weakly to a probability $Q$ on $C(\IR_+,C)$, and the claim (\ref{5.1}) will follow if we can show that $Q = \wt{P}$.

\medskip
With this in mind, we first collect so some a priori properties of such a $Q$. First, the laws $\wt{P}^{\ve_n}, n \ge 0$, being stationary, we see that
\begin{equation}\label{5.2}
\mbox{$Q$ is a stationary law on $C(\IR_+,C)$}.
\end{equation}
Then, one notes that the law of $W_0 = (U_0,V_0)$ under $Q$ is the weak limit of the laws of $(U_0,V_0)$ on $C$ under $\Pen$. But, by Proposition \ref{prop3.3}, these laws are tight on $\stackrel{_\circ}{C}$, so that
\begin{equation}\label{5.3}
\mbox{under $Q$ the law of $W_0$ (or of any $W_t, t \ge 0$, by (\ref{5.2})) gives full mass to $\stackrel{_\circ}{C}$}.
\end{equation}
As we now explain, Theorem \ref{theo5.1} will be established if we can show that
\begin{equation}\label{5.4}
\begin{array}{l}
\mbox{under $Q$, for any smooth compactly supported function $\psi$ on $\stackrel{_\circ}{C}$},
\\
\mbox{$\psi(W_t) - \psi(W_0) - \dis\int^t_0 \wt{A} \psi (W_s) \,ds, s \ge 0$, is an $(\cF_t)_{t \ge 0}$-martingale}.
\end{array}
\end{equation}
Indeed, using the Lyapunov-Foster function in (3.21) of \cite{SzniWidm26a} and Theorem 2.1 on p.~524 of \cite{MeynTwee93}, it will follow from (\ref{5.3}) that $Q$ gives full mass to $C(\IR_+,\stackrel{_\circ}{C})$. In addition, $Q$ will satisfy (\ref{2.16}) and hence coincide with $\wt{P}$.

\medskip
Thus, there remains to establish (\ref{5.4}). We will now reduce this task to the proof of (\ref{5.10}) below.

\medskip
First, note that (\ref{5.4}) will follow if we show that for any $0 \le s_1 \le \dots \le s_m \le t < t^\prime$, any continuous compactly supported functions $\varphi_1,\dots,\varphi_m$ on $\IR^2$ bounded by $1$ in absolute value, and any smooth compactly supported function $\psi$ on $\stackrel{_\circ}{C}$, one has (with $E^Q$ denoting the $Q$-expectation):
\begin{equation}\label{5.5}
E^Q \Big[\big(\psi(W_{t^\prime}) - \psi(W_t) - \dis\int^{t^\prime}_ t \! \wt{A} \psi (W_s) \,ds\big) \prod\limits_{k=1}^m \varphi_k(W_{s_k})\Big] = 0.
\end{equation}
Since $\wt{P}^{\ve_n}$ converges weakly to $Q$ and $\wt{P}^{\ve_n}$ is the law on $C(\IR_+,C)$ of $(W_t)_{t \ge 0}$ under $\Pen$ (on $C(\IR_+, \IR^N)$), the equality (\ref{5.5}) will follow if we show that (with $\Een$ the $\Pen$-expectation):
\begin{equation}\label{5.6}
\lim\limits_n \Een \Big[\big(\psi(W_{t^\prime}) - \psi(W_t) - \dis\int^{t^\prime}_t \! \wt{A} \psi(W_s) \, ds\big) \, F\Big] = 0, \; \mbox{with} \; F = \prod\limits_{k=1}^m \varphi_k (W_{s_k}).
\end{equation}
By construction, see (\ref{1.25}), (\ref{1.29}), the expectation where we replace in (\ref{5.6}) $\wt{A} \psi(W_s)$ by $\Le g(X_s)$, with $g(x) = \psi (|x|^2, |x|^2_{-1})$, is equal to zero. Noting that $\Le g = \wt{L} g$ by (\ref{1.22}), the claim (\ref{5.5}) will follow once we show that
\begin{equation}\label{5.7}
\lim\limits_n \Een \Big[\dis\int^{t^\prime}_t \wt{L} g (X_s) - \wt{A} \psi(W_s) \, ds \,F\Big] = 0.
\end{equation}
Taking into account the expression for $\wt{L} g(x)$ on the second and third line of (\ref{1.22}), and the definition of $\wt{A} \psi$ in (\ref{2.10}), the claim (\ref{5.5}) will thus follow if for any smooth function $h$ compactly supported in $\stackrel{_\circ}{C}$, $0 \le t < t^\prime$, $F$ as in (\ref{5.6}), and $1 \le \ell \le N$, one has
\begin{equation}\label{5.8}
\lim\limits_n \Een \Big[\dis\int^{t^\prime}_t  \big(X^2_{s,\ell} - q_\ell (W_s)\big) \, h(W_s)\,ds \,F\Big] = 0.
\end{equation}
Now, if $\tau'_0 = t < \tau'_1 < \dots < \tau'_J = t^\prime$ with $J \ge 1$ are intermediate points with a constant increment $\tau'_{j+1} - \tau'_j = (t^\prime - t) / J$, for $0 \le j < J$, we note that using stationarity, and Cauchy-Schwarz Inequality, for $0 < \ve < 1$ one has
\begin{equation}\label{5.9}
\begin{array}{l}
\Ee \Big[\Sigma_{0 \le j < J} \dis\int^{\tau'_{j+1}}_{\tau'_j} |X^2_{s,\ell} - q_\ell (W_s)| \; |h(W_s) - h(W_{\tau'_j})| \,ds\Big] \le
\\
J \dis\int_0^{(t^\prime - t)/J} ds\,\Ee [ | X^2_{s,\ell} - q_\ell (W_s)|^2]^{1/2} \; \Ee [ |h(W_s) | - h(W_0)|^2]^{1/2} \stackrel{(\ref{3.6}),(\ref{A.21})}{\le}
\\[2ex]
c\,(t^\prime - t) \sup\limits_{0 \le s \le (t^\prime - t) / J} \Ee [ | h(W_s) - h(W_0)|^2]^{1/2},
\end{array}
\end{equation}
and the expression with the last line tends uniformly in $0 < \ve \le 1$ to zero as $J$ goes to infinity by the tightness of the laws $\wt{P}^\ve$ of $W_\point$, see Theorem \ref{theo4.1}. As a result, (\ref{5.8}) will follow once we can show that for any fixed $J \ge 1$, a similar vanishing limit as in (\ref{5.8}) holds for the $\Pen$-expectation of $\Sigma_{0 \le j < J} \int^{\tau'_{j+1}}_{\tau'_j} (X^2_{s,\ell} - q_\ell (W_s)) \, ds \, h(W_{\tau'_j})\,F$.

\medskip
Therefore, we now see that to prove (\ref{5.4}) it suffices to show that
\begin{equation}\label{5.10}
\left\{ \begin{array}{l}
\mbox{for any $0 \le s_1 \le \dots \le s_m \le t < t^\prime$, continuous compactly supported functions}
\\
\mbox{$\varphi_1, \dots, \varphi_m$ on $\IR^2$ bounded in absolute value by $1$, and $1 \le \ell \le N$, one has}
\\
\mbox{$\lim\limits_n \Een \big[\int^{t^\prime}_t X^2_{s,\ell} - q_\ell (W_s) \, ds \, F\big] = 0$, where}
\\
F = \varphi_1(W_{s_1}) \dots \varphi_m(W_{s_m}).
\end{array}\right. 
\end{equation}
This is the promised reduction for the proof of (\ref{5.4}).

\medskip
We will now prove (\ref{5.10}). For this purpose it is convenient to use the stochastic differential equation formulation for the diffusion process $X_s, s \ge 0$, see (\ref{1.31}). So, $X_s, s \ge 0$, under $\Pe$, with $\ve > 0$, can be represented on a possibly extended filtered probability space as the solution of the Stratonovich differential equation
\begin{equation}\label{5.11}
\begin{split}
d X(t) = &  \; \Big(\!\!- \Lambda X(t) + \mbox{\f $\dis\frac{1}{\ve}$} \; b (X(t))\Big) \, dt + \Sigma^N_1 \Big(\lambda_\ell \,a (1 + \delta_\ell)\Big)^{\frac{1}{2}} e_\ell \, d \beta_\ell(t) 
\\[-1ex]
& \;\; + \Big(\mbox{\f $\dis\frac{\kappa}{\ve}$}\Big)^{\frac{1}{2}} \Sigma^M_1 \, Z_m (X(t)) \circ d \wt{\beta}_m(t), \mbox{with the initial condition $X(0)$},
\end{split}
\end{equation}
where $e_\ell, 1 \le \ell \le N$, stands for the canonical basis of $\IR^N$, $\Lambda$ for the $N \times N$ diagonal matrix with coefficients $\lambda_\ell, 1 \le \ell \le N$, and $\beta_\ell, 1 \le \ell \le N$, $\wt{\beta}_m, 1 \le m \le M$, are independent Brownian motions for the filtration still denoted by $(\cG_t)_{t \ge 0}$, and we write $X(t)$ in place of $X_t$.

\medskip
We now wish to prove (\ref{5.10}). We thus choose $0 \le s, \le \dots \le s_m \le t < t^\prime$, $\varphi_1, \dots, \varphi_m$ continuous compactly supported functions on $\IR^2$ bounded in absolute value $1$, and $\ell$ in $\{1, \dots, N\}$ and will prove, with $F$ as in (\ref{5.10}), that
\begin{equation}\label{5.12}
H \stackrel{\rm def}{=} \Ee \,\Big[ \dis\int^{t^\prime}_t X^2_\ell(s) - q_\ell(W_s) \,ds \, F\Big] \rightarrow 0, \; \mbox{as $\ve \r 0$}.
\end{equation}
We tacitly assume that $0 < \ve \le 1$ from now on. We choose $u_{\min}, u_{\max}$, and $\eta$ with
\begin{equation}\label{5.13}
0 < u_{\min} < u_{\max}, \; u_{\max} > 1, \;\eta > 0
\end{equation}
and recall the definition of the {\it good} set $G(u_{\min}, u_{\max}, \eta)$ in (\ref{3.23}), and its complement, the {\it untamed} subset $U(u_{\min}, u_{\max}, \eta)$ in (\ref{3.16}); both are subsets of $\IR^N$. We then introduce dependent on $\ve$
\begin{equation}\label{5.14}
\mbox{$K(\ve) > 0$ integer such that for some $\nu \in \Big(0, \fr\Big)$, $K(\ve) \sim \ve^{-1} \Big(\log \mbox{\f $\dis\frac{1}{\ve}$}\Big)^{-\nu}$, as $\ve \r 0$}
\end{equation}
(the constraint on $\nu$ will be important in (\ref{5.45})), as well as the intermediate times (not to be confused with the intermediate times below (\ref{5.8}), and with $K$ shorthand notation for $K(\ve)$):
\begin{equation}\label{5.15}
\mbox{$t_0 = t < t_1 < \dots < t_K = t^\prime$ with $t_{k+1} - t_k = (t^\prime - t) / K \stackrel{\rm def}{=} \tau$, for $0 \le k < K$}, 
\end{equation}
so that
\begin{equation}\label{5.16}
\mbox{$\ve \le \tau \r 0$ and $\ve / \tau \r 0$, as $\ve \r 0$}.
\end{equation}
We now come back to the quantity $H$ of interest in (\ref{5.12}) and write:
\begin{equation}\label{5.17}
\begin{split}
H = H_1 + H_2, \; \mbox{where} &\; H_1 = \Ee \,\Big[\Sigma_{0 \le k < K} \dis\int^{t_{k+1}}_{t_k} X^2_\ell (s) - q_\ell (W_{t_k}) \,ds \, F\Big] \; \mbox{and}
\\
&\;H_2 = \Ee \,\Big[\Sigma_{0 \le k < K} \dis\int^{t_{k+1}}_{t_k} q_\ell (W_{t_k}) - q_\ell(W_s) \,ds \, F\Big].
\end{split}
\end{equation}
As we now explain
\begin{equation}\label{5.18}
\lim\limits_{\ve \r 0} \;H_2 = 0.
\end{equation}
Indeed, for any $A > 0$, we have $q_\ell = q_\ell - q_\ell \wedge A + q_\ell \wedge A$, so that (recall $|F| \le 1$)
\begin{equation}\label{5.19}
\begin{split}
|H_2| \le & \; \Ee\, \Big[ \Sigma_{0 \le k < K} q_\ell (W_{t_k}) \,1\{q_\ell(W_{t_k}) \ge A\}  (t_{k+1} - t_k) + \dis\int^{t^\prime}_t \!\!\! q_\ell (W_s) \, 1\{q_\ell (W_s) \ge A\}\Big] \, ds
\\
& \qquad \! + (t^\prime -t ) \, \Ee [\sup\limits_{0 \le s, s^\prime \le \tau} | \,q_\ell \wedge A(W_s) - q_\ell \wedge A (W_{s^\prime})|]
\\
= & \; 2 (t^\prime - t) \, \Ee \,[q_\ell (W_0) \, 1 \{q_\ell(W_0) \ge A\}] 
\\
& \qquad \! + (t^\prime - t) \, \Ee [\sup\limits_{0 \le s, s^\prime \le \tau}  | \,q_\ell \wedge A(W_s) - q_\ell \wedge A (W_{s^\prime})|],
\end{split}
\end{equation}
using stationarity in the second and third line. By the tightness of the laws of $W_s, s \ge 0$, under $\Pe$ shown in Theorem \ref{4.1} and by (\ref{5.16}), the last term of (\ref{5.19}) goes to $0$ as $\ve \r 0$. Further, by the exponential moment bound on $U_0$ in (\ref{3.6}), and the bound $q_\ell (u,v) \le u$, see (\ref{2.4}), the quantity on the third line of (\ref{5.19}) goes to $0$ as $A$ tends to infinity. This proves (\ref{5.18}).

\medskip
We now consider $H_1$ in (\ref{5.17}). For each $0 \le k < K$ we introduce the stochastic process $z_k(s)$, $t_k \le s \le t_{k + 1}$, solution of (see (\ref{1.10}), (\ref{1.16}) for notation)
\begin{equation}\label{5.20}
\begin{split}
d z_k(s) & = \mbox{\f $\dis\frac{1}{\ve}$} \; b\big(z_k(s)\big)\, ds + \Big(\mbox{\f $\dis\frac{\kappa}{\ve}$}\Big)^{1/2} \; \Sigma^M_1\, Z_m \big(z_k(s)\big) \circ d \wt{\beta}_m(s),
\\
z_k(t_k) & = X(t_k) .
\end{split}
\end{equation}
This corresponds to a similar stochastic differential equation (on the time interval $[t_k, t_{k+1}]$) as in (\ref{5.11}), where only the ``fast terms'', where $\ve$ appears, are kept, and with an initial condition given by $X(t_k)$. Since $b$ and $Z_m, 1 \le m \le M$, satisfy (\ref{1.11}), (\ref{1.17}), and (\ref{5.20}) is a Stratonovich differential equation, $z_k(\cdot)$ preserves the $|\cdot |$ and $| \cdot |_{-1}$ norm (and is non-explosive):
\begin{equation}\label{5.21}
\mbox{$|z_k(s)|^2 = U_{t_k}$ and $|z_k(s)|^2_{-1} = V_{t_k}$, for all $s \in [t_k, t_{k+1}]$}.
\end{equation}
We then write (with $z_{k,\ell}(\cdot)$ denoting the $\ell$-th component of $z_k (\cdot)$):
\begin{equation}\label{5.22}
\begin{split}
H_1 = I_1 +I_2, \; \mbox{where} &\; I_1 = \Ee \,\Big[\Sigma_{0 \le k < K} \dis\int^{t_{k+1}}_{t_k} X^2_\ell (s) - z^2_{k,\ell} (s) \,ds \, F\Big] \; \mbox{and}
\\
&\;I_2 = \Ee \,\Big[\Sigma_{0 \le k < K} \dis\int^{t_{k+1}}_{t_k} z^2_{k,\ell} (s) - q_\ell (W_{t_k}) \,ds \, F\Big] .
\end{split}
\end{equation}
We first handle $I_2$. For each $k$, we consider whether $X(t_k)$ is ``good'', that is $X(t_k) \in G(u_{\min},u_{\max},\eta)$, see (\ref{3.23}), otherwise we say that it is ``bad''. We then write
\begin{equation}\label{5.23}
\begin{split}
I_2 =I_{2,1} + I_{2,1}, \; \mbox{where} &\; I_{2,1} = \Ee \Big[\Sigma_{0 \le k < K} \dis\int^{t_{k+1}}_{t_k} \!\! z^2_{k,\ell} (s) - q_\ell (W_{t_k}) \,ds \, F, X(t_k) \; \mbox{is good}\Big], 
\\
&\;I_{2,2} = \Ee \Big[\Sigma_{0 \le k < K} \dis\int^{t_{k+1}}_{t_k} \!\!  z^2_{k,\ell} (s) - q_\ell (W_{t_k}) \,ds \, F, X(t_k) \; \mbox{is bad}\Big].
\end{split}
\end{equation}
We first bound $I_{2,2}$. By (\ref{5.20}), (\ref{5.21}) and stationarity we see that
\begin{equation}\label{5.24}
\begin{split}
|I_{2,2}| & \le (t^\prime -t ) \,\Ee \,\big[\big(U_0 + q_\ell (W_0)\big), \; X(0) \notin G(u_{\min},u_{\max}, \eta)\big]
\\
& \le c\,(t^\prime - t) \, \gamma	({u_{\min},u_{\max}, \eta})^{1/2},
\end{split}
\end{equation}
where we used the Cauchy-Schwarz Inequality, (\ref{3.6}), $q_\ell(w) \le u$, for $w = (u,v)$ by (\ref{2.4}), as well as (\ref{3.23}), (\ref{3.17}).

\medskip
We now bound $I_{2,1}$, see (\ref{5.23}). We use the notation $\mu_w$ for $\mu_{u,v}$, with $(u,v) \in C$, see (\ref{A.7}), which is the conditional law of $\mu$ given $|x|^2 = u$ and $|x|^2_{-1} = v$ (see (\ref{A.12})), and we recall the notation $p_t(\cdot,\cdot)$ for the transition density with respect to $\mu_w$ of the semi-group generated by $B + \kappa \cD$, when $w = (u,v) \in \, \stackrel{_\circ}{C}$ and $u/v$ does not belong to $\{\lambda_\ell, 1 \le \ell \le N\}$, see (\ref{B.11}). Note that ``$X(t_k)$ good'', see above (\ref{5.23}), means that $X(t_k)$ belongs to $G(u_{\min},u_{\max}, \eta)$, and hence $W_{t_k}$ fulfills these conditions (and even the assumptions of Proposition \ref{propB.2}). We thus have
\begin{equation}\label{5.25}
\begin{split}
I_{2,1} & = \Sigma_{0 \le k < K} \dis\int^{t_{k+1}}_{t_k} ds \, \Ee \, \big[\big(z^2_{k,\ell}(s) - q_\ell(W_{t_k})\big) \, F, X(t_k) \; \mbox{is good}\big]
\\
& = \Sigma_{0 \le k < K} \dis\int^\tau_0 ds \, \Ee \, \Big[ \dis\int \big(p_{s/\ve}\big(X(t_k),z\big) - 1\big) \, z^2_\ell \, d\mu_{W_{t_k}}(z) \; F, X(t_k) \; \mbox{is good$\Big]$},
\end{split}
\end{equation}
where we have used (\ref{5.15}), conditioning on $\cG_{t_k}$, and the identity (\ref{A.21}) to express $q_\ell(W_{t_k})$ as a $\mu_{W_{t_k}}$-integral.

\medskip
On the event where $X(t_k)$ is good we have $z^2_\ell \le |z|^2 \le u_{\max}$, $d\mu_{W_k}(z)$-a.s., and using (\ref{3.24}) (see also Proposition \ref{propB.2}) when $s \ge \ve$, we obtain
\begin{equation}\label{5.26}
\begin{split}
|I_{2,1}| & \le K \,u_{\max} \Big(2 \ve + \dis\int^\tau_\ve c_2 \, \exp\Big\{ - c_3 \;\mbox{\f $\dis\frac{s}{\ve}$}\Big\} \; ds\Big)
\\
& \le \ve \, K \,u_{\max} \,  \Big(2 + \mbox{\f $\dis\frac{c_2}{c_3}$}\Big)
\end{split}
\end{equation}
where $c_2 (u_{\min},u_{\max},\eta)$ and $c_3 (u_{\min},u_{\max},\eta)$ are the constants from Proposition \ref{propB.2}. By (\ref{5.14}) the product $\ve \,K$ tends to $0$ as $\ve$ goes to $0$, so that
\begin{equation}\label{5.27}
\lim\limits_{\ve \r 0} \; |I_{2,1}| = 0
\end{equation}
and, combined with (\ref{5.24}), (\ref{5.23}), we find 
\begin{equation}\label{5.28}
\limsup\limits_{\ve \r 0} \; |I_2| \le   c(t^\prime - t) \,\gamma(u_{\min},u_{\max},\eta)^{1/2}.
\end{equation}
We now turn to $I_1$ in (\ref{5.22}) and introduce for $0 \le k  < K$ the $(\cG_t)$-stopping times (see (\ref{5.15}), (\ref{5.13}) for notation):
\begin{equation}\label{5.29}
\tau_k = \inf\{ s \ge t_k; U_s \ge u_{\max}\}
\end{equation}
and we write
\begin{equation}\label{5.30}
\begin{split}
I_1 =I_{1,1} + I_{1,2}, \; \mbox{where} &\; I_{1,1} = \Ee\, \Big[\Sigma_{0 \le k < K} \dis\int^{t_{k+1}}_{t_k} \!\! X^2_\ell (s \wedge \tau_k) -z^2_{k,\ell} (s \wedge \tau_k) \,ds \, F\Big], \;\
\\
&\qquad \;\; \; \mbox{and}
\\
&\;I_{1,2} = \Ee\, \Big[\Sigma_{0 \le k < K} \dis\int^{t_{k+1}}_{t_k} \!\! X^2_\ell (s) - X^2_\ell (s \wedge \tau_k) - z^2_{k,\ell}(s) + z^2_{k,\ell}(s \wedge \tau_k)\, ds \, F\Big].\end{split}
\end{equation}
We first bound $I_{1,2}$.

\medskip
We note that the integral inside the expectation vanishes if $\tau_k \ge t_{k+1}$, and is bounded by $\int^{t_{k+1}}_{t_k}U_s \, ds + 3(U_{t_k} + u_{\max})(t_{k+1} - t_k)$ otherwise. We thus find that (recall that $|F| \le 1$)
\begin{equation}\label{5.31}
|I_{1,2}| \le \Sigma_{0 \le k  < K} \dis\int^{t_{k+1}}_{t_k} \Ee\, [U_s + 3 U_{t_k} + 3 u_{\max}, t_k \le \tau_k < t_{k+1}]\, ds.
\end{equation}
Applying the Cauchy-Schwarz Inequality and (\ref{3.6}), we thus find that
\begin{equation}\label{5.32}
|I_{1,2}| \le c\,(1 + u_{\max}) (t^\prime - t) \; \Pe [\tau_0 < t_1]^{1/2}.
\end{equation}
Next, we note that when $t_1 = \tau \le 1$, i.e.~when $\ve$ is small enough, see (\ref{5.15}), as we will from now on assume, one has
\begin{equation}\label{5.33}
\begin{split}
\Pe\,[\tau_0 < t_1] & \le \Pe\,\Big[U_0 \ge \mbox{\f $\dis\frac{u_{\max}}{2}$} \Big] + \Pe \, \Big[\sup\limits_{0 \le s \le \tau} |U_s - U_0| \ge \mbox{\f $\dis\frac{u_{\max}}{2}$} 
\Big]
\\
&\!\!\!\! \stackrel{(\ref{4.10})}{\le} \Pe \,\Big[U_0 \ge \mbox{\f $\dis\frac{u_{\max}}{2}$} \Big] + c \, \tau^2 / u^4_{\max} .
\end{split}
\end{equation}
Thus, coming back to (\ref{5.32}), we find that
\begin{equation}\label{5.34}
\overline{\lim\limits_{\ve \r 0}} \;\; |I_{1,2}| \le c\,(t^\prime - t) (1 + u_{\max}) \; \sup\limits_{0 < \ve \le 1} \Pe\, \Big[U_0 \ge \mbox{\f $\dis\frac{u_{\max}}{2}$} \Big]^{1/2}.
\end{equation}
We now turn to $I_{1,1}$, see (\ref{5.30}). We will use the martingale property satisfied by $(X(s), z_k(s))$, $s \in [t_k,t_{k+1}]$. We introduce the functions, see (\ref{1.10}), (\ref{1.19})
\begin{equation}\label{5.35}
\mbox{$f_\ell(x) = x^2_\ell$ and $h_\ell(x) = (B + \kappa \cD) \, f_\ell (x)$ for $x \in \IR^N$}.
\end{equation}
Then we find, using the martingale property and $X(t_k) = z_k(t_k)$, that
\begin{equation}\label{5.36}
\begin{split}
I_{1,1} & = \Sigma_{0 \le k < K} \dis\int^{t_{k+1}}_{t_k} \Ee\, \big[\big(X^2_\ell (s \wedge \tau_k) - z^2_{k,\ell} (s \wedge \tau_k)\big)\, F\big] \, ds
\\
& =  \Sigma_{0 \le k < K} \dis\int^{t_{k+1}}_{t_k} \Ee\, \Big[ \dis\int^{s \wedge \tau_k}_{t_k} \Big\{\wt{L} + \mbox{\f $\dis\frac{1}{\ve}$} \;(B + \kappa \cD) \, f_\ell \big(X(\sigma)\big)
\\
& \hspace{3.5cm}- \mbox{\f $\dis\frac{1}{\ve}$} \, (B + \kappa \cD) \, f_\ell  \big(z_k(\sigma)\big)  \Big\} \, d \sigma \, F\Big] \,ds
\\
& = J_1 + J_2,
\end{split}
\end{equation}
where
\begin{equation}\label{5.37}
\begin{split}
J_1& = \Sigma_{0 \le k < K} \dis\int^{t_{k+1}}_{t_k} \Ee\,\Big[\dis\int^{s \wedge \tau_k}_{t_k} \wt{L} f_\ell\big(X(\sigma)\big) \, d \sigma \, F\Big] \; \mbox{and}
\\
J_2& = \Sigma_{0 \le k < K} \dis\int^{t_{k+1}}_{t_k} \Ee\,\Big[\dis\int^{s \wedge \tau_k}_{t_k} \,h_\ell\big(X(\sigma)\big) - h_\ell \big(z_k(\sigma)\big) \, d \sigma \, F\Big]\, ds.
\end{split}
\end{equation}
We fist bound $J_1$. We note that by (\ref{1.9}) and (\ref{5.35})
\begin{equation}\label{5.38}
| \wt{L} \, f_\ell (x) | \le  c\, (x_\ell^2 + 1),
\end{equation}
and hence that
\begin{equation}\label{5.39}
|J_1| \le \Sigma_{0 \le k < K} \dis\int^{t_{k+1}}_{t_k} c\,(s - t_k) \, \Ee [ X^2_\ell(0) + 1] \, ds \stackrel{(\ref{5.15})}{\le} c\,(t^\prime - t) \, \tau( \Ee [U_0] + 1).
\end{equation}
With the identity (\ref{3.4}) and (\ref{5.16}) we thus see that 
\begin{equation}\label{5.40}
\lim\limits_{\ve \r 0} \; J_1 = 0.
\end{equation}
We proceed with $J_2$. We will apply the Gronwall Inequality to bound the difference $|X(\sigma) - z_k(\sigma)|$. We first note that inside the ball of $\IR^N$ where $|x|^2 \le u_{\max}$, the $\IR^N$-valued functions $b(\cdot), Z_m(\cdot), \Sigma_\ell \, Z_{m,\ell}(\cdot) \, \partial_\ell \, Z_{m,\ell}(\cdot)$, see (\ref{1.10}), (\ref{1.16}), satisfy a Lipschitz property with a constant depending on $u_{\max}$ (we refer to the beginning of Section 3 for our convention concerning constants). Further, as noted in (\ref{5.21}), $\IP_\ve$-a.s., $|z_k(s)| = |z_k(t_k)|$, for all $s$ in $[t_k,t_{k+1}]$.

\medskip
Thus, applying Doob's Inequality, see \cite{KaraShre88}, p.~14, and noting that $|X|(\wt{\sigma})|^2$ and $|z_k(\wt{\sigma})|^2$ are at most $u_{\max}$ when $\tau_k > t_k$ and $t_k \le \wt{\sigma} \le \tau_k$, and that $\{\tau_k = t_k\} = \cG_{t_k}$, we find by (\ref{5.11}) and (\ref{5.20}) that for $0 \le k < K$ and $t_k \le s \le t_{k+1}$:
\begin{equation}\label{5.41}
\begin{array}{l}
\Ee [\sup\limits_{t_k \le \sigma \le s} \, | \, X(\sigma \wedge \tau_k) - z_k (\tau \wedge \tau_k)|^2]  \le C\, \Ee \Big[\Big(\dis\int^s_{t_k} | X(\wt{\sigma} \wedge \tau_k) | \, d\wt{\sigma}\Big)^2 1_{\{\tau_k > t_k\}}
\\[1ex]
+\; \Sigma_\ell \sup\limits_{\sigma \in [t_k,s]} |\beta_\ell (\sigma) - \beta_\ell(t_k)|^2\Big] +  \Big[ \mbox{\f $\dis\frac{c(u_{\max})}{\ve}$} \Big\{ \dis\int^s_{t_k} \Ee [ | X(\wt{\sigma} \wedge \tau_k) - z_k(\wt{\sigma} \wedge \tau_k)| ] \, d \wt{\sigma}\Big\}^2 
\\[1ex]
+ \;\mbox{\f $\dis\frac{c(u_{\max})}{\ve}$} \dis\int^s_{t_k} \Ee [ | X(\wt{\sigma} \wedge \tau_k) - z_k(\wt{\sigma} \wedge \tau_k)|^2 ]\,  d \wt{\sigma}\Big]  \le c(u_{\max}) (s - t_k) (s- t_k + 1) 
\\[1ex]
+ \; c(u_{\max}) \Big(\mbox{\f $\dis\frac{s-t_k}{\ve^2}$}  + \mbox{\f $\dis\frac{1}{\ve}$}\Big) \dis\int^s_{t_k} \Ee[| X(\sigma \wedge \tau_k) - z_k (\sigma \wedge \tau_k) |^2] \, d\tau .
\end{array}
\end{equation}
Applying the Gronwall Inequality, we find that for $s \in [t_k, t_{k+1}]$
\begin{equation}\label{5.42}
\begin{array}{l}
\Ee [\sup\limits_{t_k \le \sigma \le s} \, | \, X(\sigma \wedge \tau_k) - z_k (\sigma \wedge \tau_k)|^2]  \le 
\\
c(u_{\max}) (s -t_k)(s - t_k + 1) \, \exp\Big\{ \mbox{\f $\dis\frac{c(u_{\max})}{\ve}$} \;\Big(\mbox{\f $\dis\frac{s - t_k}{\ve}$} + 1\Big) (s - t_k)\Big\}.
\end{array}
\end{equation}
In particular, choosing $s = t_{k+1}$, with (\ref{5.15}), we obtain (recall $\tau$ from (\ref{5.15})):
\begin{equation}\label{5.43}
\begin{array}{l}
\Ee [\sup\limits_{t_k \le \sigma \le t_{k+1}} \, | \, X(\sigma \wedge \tau_k) - z_k (\sigma \wedge \tau_k)|^2]  \le 
\\
c(u_{\max})\,\tau(\tau + 1) \, \exp\Big\{ c(u_{\max}) \;\Big(\mbox{\f $\dis\frac{\tau}{\ve}$} + 1\Big) \; \mbox{\f $\dis\frac{\tau}{\ve}$}\Big\}.
\end{array}
\end{equation}
We can now bound $J_2$, see (\ref{5.37}). We recall the notation (\ref{5.35}). We have
\begin{equation}\label{5.44}
\begin{split}
|J_2| & \le \Sigma_{0 \le k < K} \dis\int^{t_{k+1}}_{t_k} \; \mbox{\f $\dis\frac{1}{\ve}$} \; \Ee\, \Big[ \dis\int^{s \wedge \tau_k}_{t_k} \big|h_\ell\big(X(\sigma)\big) - h_\ell \big(z_k(\sigma)\big)\big| \, d \sigma\Big] \, ds
\\
&\hspace{-2.5ex} \stackrel{(\ref{5.43}),\tau \le 1}{\le} \mbox{\f $\dis\frac{c(u_{\max})}{\ve}$} \; \Sigma_{0 \le k < K} \dis\int^{t_{k+1}}_{t_k} \dis\int^{s}_{t_k} d \sigma \, ds \; \tau^{1/2} \exp \Big\{ c(u_{\max}) \Big(  \mbox{\f $\dis\frac{\tau}{\ve}$} + 1\Big) \;  \mbox{\f $\dis\frac{\tau}{\ve}$}\Big\}
\\
& = \,c(u_{\max}) (t^\prime - t) \; \mbox{\f $\dis\frac{\tau^{3/2}}{\ve}$} \; \exp \Big\{ c(u_{\max}) \Big(\mbox{\f $\dis\frac{\tau}{\ve}$} + 1\Big) \; \mbox{\f $\dis\frac{\tau}{\ve}$}\Big\}.
\end{split}
\end{equation}
By (\ref{5.14}), (\ref{5.15}), as $\ve$ tends to $0$, we have $\frac{\tau}{\ve} \sim (t^\prime - t)(\log \frac{1}{\ve})^\nu$ (where $\nu \in (0, \frac{1}{2})$), and coming back to the last line of (\ref{5.44}) we find that
\begin{equation}\label{5.45}
\lim\limits_{\ve \r 0} \; J_2 = 0.
\end{equation}
We can now collect the bounds (\ref{5.40}), (\ref{5.45}) to infer that
\begin{equation}\label{5.46}
\overline{\lim\limits_{\ve \r 0}} \;I_{1,1} = 0.
\end{equation}
Thus, combining (\ref{5.28}), (\ref{5.30}), (\ref{5.34}) with the above, we find that (see (\ref{5.12}))
\begin{equation}\label{5.47}
\begin{split}
\overline{\lim\limits_{\ve \r 0}} \; |H| \stackrel{(\ref{5.18})}{=} \overline{\lim\limits_{\ve \r 0}} \; |H_1| & \stackrel{(\ref{5.22})}{\le} \overline{\lim\limits_{\ve \r 0}} \; |I_1| + \overline{\lim\limits_{\ve \r 0}}  \; |I_2|
\\
& \;\;\;  \le c\, (t - t^\prime) \Big\{(1 + u_{\max}) \; \sup\limits_{0 < \ve \le 1} \; \Pe\Big[U_0 \ge \mbox{\f $\dis\frac{u_{\max}}{2}$}\Big] 
\\
& \quad \;\; + \gamma(u_{\min}, u_{\max}, \eta)^{\frac{1}{2}}\Big\}.
\end{split}
\end{equation}
We can now let $u_{\min}, u_{\max}, \eta$ respectively tend to $0$, infinity, and zero, so that by (\ref{3.18}) and (\ref{3.6}) we conclude that
\begin{equation}\label{5.48}
\lim\limits_{\ve \r 0} \; H = 0,
\end{equation}
and thus (\ref{5.10}) holds. As already mentioned, (\ref{5.4}) follows and this implies (\ref{5.1}). This concludes the proof of Theorem \ref{theo5.1}. \hfill $\square$

\medskip
A a direct consequence of Theorem \ref{theo5.1} and Proposition \ref{prop3.1} we have, with the notation of (\ref{3.2}), (\ref{3.3}), the following:
\begin{corollary}\label{cor5.3}
If $\psi$ is a continuous real valued function on $C$, which is such that for some $z_0 \in (0,1/B^\prime_0)$, $z_1 \in (0, B^\prime_1)$
\begin{equation}\label{5.49}
\mbox{$\psi (u,v) / (e^{z_0 v} + e^{z_1 u})$ is a bounded function on $C$},
\end{equation}
then with $\wt{\pi}$ as in (\ref{2.12})
\begin{equation}\label{5.50}
\lim\limits_{\ve \r 0} \; \Ee \,[\psi(W_0)] = \dis\int \psi(w) \, \wt{\pi} (dw).
\end{equation}
\end{corollary}

\begin{proof}
By Theorem \ref{theo5.1} the law on $C$ of $W_0$ under $\wt{\IP}_\ve$ converges weakly to $\wt{\pi}$. In addition, by Proposition \ref{prop3.1}, for $z^\prime_0 \in (z_0, 1/B^\prime_0)$, and $z^\prime_1 \in (z_1, 1/B^\prime_1)$, $\Ee[e^{z^\prime_0 V_0} + e^{z^\prime_1 U_0}]$ remains uniformly bounded in $\ve > 0$. This classically readily implies (\ref{5.50}).
\end{proof}

\begin{remark}\label{rem5.4} \rm
The same statement as above holds if the fixed $\kappa \in (0,1]$ entering the definition of $\Pe$ in (\ref{1.29}) is replaced by a $\kappa_\ve \in (0,1]$ tending to $0$ as $\ve$ goes to $0$, for which the corresponding $\Pe$ weakly converges to $\wt{P}$ as $\ve \r 0$. As mentioned in Remark \ref{rem5.2}, Theorem \ref{theo5.1} ensures that one can construct some such examples of  $\kappa_\ve$. \hfill $\square$
\end{remark}

\section{Inviscid condensation bound}
\setcounter{equation}{0}

In this section we use the weak convergence result of the previous section and the condensation bound (\ref{2.19}), see also Theorem \ref{theo5.1} of the companion article \cite{SzniWidm26a}, to obtain a quantitative lower bound on the inviscid limit of the ratio $\Ee [V_0] / \Ee[U_0]$.

\medskip
We recall that $B_0 = a \, \Sigma_\ell (1 + \delta_\ell) < B_1 = a \, \Sigma_\ell \, \lambda_\ell (1 + \delta_\ell)$, see (\ref{3.2}), (\ref{3.3}), and that $2 \, \Ee[U_0] = B_0$ for all $\ve > 0$, see (\ref{3.4}).
\begin{theorem}\label{theo6.1} (Inviscid condensation bound)

\medskip\n
For any $\ell_0$ in $\{3, \dots, N\}$ (recall $N = 2n$), writing $\ell_0 = 2i_0$ or $\ell_0 = 2i_0 -1$, with $2 \le i_0 \le n$, depending on whether $\ell_0$ is even or odd, one has
\begin{equation}\label{6.1}
2 \lim\limits_{\ve \r 0} \; \Ee[U_0 - V_0] \le \mbox{\f $\dis\frac{B_1 - B_0}{\lambda_{\ell_0} - 1}$} + \mbox{\f $\dis\frac{\lambda_3}{\lambda_3 - 1}$} \; \sum\limits^{i_0 - 2}_{k=1} \;\mbox{\f $\dis\frac{1}{n-k}$} \; B_0 \le \mbox{\f $\dis\frac{B_1 - B_0}{\lambda_{\ell_0} - 1}$} +  \mbox{\f $\dis\frac{\lambda_3}{\lambda_3 - 1}$}  \;  \mbox{\f $\dis\frac{\ell_0}{N - \ell_0}$} \; B_0
\end{equation}
(where the limit in the left member exists by Corollary {\rm \ref{cor5.3}} and {\rm (\ref{3.6})}, and the sum in the middle member is omitted when $i_0 = 2$).
\end{theorem}

\begin{proof}
This is a direct application of Corollary \ref{cor5.3} and (\ref{2.19}) (see also Theorem 5.1 of the companion article \cite{SzniWidm26a}).
\end{proof}

\begin{remark}\label{rem6.2} \rm 1) Note that when $\ell_0$ can be chosen so that $B_1/B_0 \ll \lambda_{\ell_0}$ and $\ell_0 \ll N$,  the left member of (\ref{6.1}) is small compared to $B_0 = 2 \, \Ee[U_0]$. One can compare the left member of (\ref{6.1}) to the corresponding expectation when ``$\ve = \infty$'', i.e.~when $\Le$ in (\ref{1.21}) is replaced by $\wt{L}$ in (\ref{1.9}). In this case, from the observation mentioned below (\ref{1.9}) (with hopefully obvious notation), one has
\begin{equation}\label{6.2}
2 \wt{\IE}\,\!^{\ve = \infty} [U_0 - V_0] = a \, \Sigma_\ell (1 - 1/\lambda_\ell) (1 + \delta_\ell) = B_0 (1 - B_{-1} / B_0), \; \mbox{where} \; B_{-1} = a\, \Sigma_\ell \, \lambda_\ell^{-1} (1 + \delta_\ell).
\end{equation}
The ratios $\lambda_f = B_1/B_0 \ge \wt{\lambda}_f = B_0 / B_{-1}$ can be viewed as ways of measuring (somewhat in the spirit of the ``centroid'' considerations on pp.~22,23 of \cite{Naza11}) effective spectral values for the Brownian forcing acting on the stationary diffusion $(X_t)_{t \ge 0}$ under $\Pe, \ve > 0$. Thus, when $\lambda_f \ll \lambda_{\ell_0}$ with $\ell_0 \ll N$, and $\wt{\lambda}_f$ remains bounded away from $1$, (\ref{6.1}) and (\ref{6.2}) can be viewed as a ``condensation effect'' due to the inviscid limit, which induces an attrition of all but the low modes (i.e.~$\ell = 1,2$).

\medskip\n
2) In view of Remark \ref{rem5.4} the same statement as in Theorem \ref{theo6.1} holds when $\kappa_\ve \in (0,1]$ tends to $0$ as $\ve \r 0$, and is such that the corresponding $\wt{P}_\ve$ weakly converges to $\wt{P}$ as $\ve \r 0$. From Remark \ref{rem5.2} one knows that such choices of $\kappa_\ve, \ve \r 0$, are possible. \hfill $\square$
\end{remark}

\begin{appendix}
\section{Appendix: The spaces $\IX_{u,v}$}
\setcounter{equation}{0}

In this appendix we collect several facts concerning the spaces $\IX_{u,v}$, see (\ref{A.1}) below. They will be helpful when analyzing the convergence to equilibrium of the diffusion induced by drift and stirring on these spaces in Appendix B. We refer to Section 1 for notation. In this appendix, positive constants will tacitly depend on $N$ and the $\lambda_\ell, 1 \le \ell \le N$, see (\ref{1.1}), (\ref{1.2}). Additional dependence will appear in the notation.

\medskip
We first recall some notation. We denote by $C$ the two-dimensional cone in (\ref{2.1}) and by $\stackrel{_\circ}{C}$ its interior. Of special interest in this appendix are the spaces
\begin{equation}\label{A.1}
\mbox{$\IX_{u,v} = \{x \in \IR^N$; $|x|^2 = u, |x|^2_{-1} = v\}$, for $(u,v) \in C$}.
\end{equation}
We recall, see Appendix A of \cite{SzniWidm26a}, the definition of the polytope in $\IR_+^\Sigma$ (with $\Sigma = \{3, \dots,n\}$, see also Figure 1):
\begin{equation}\label{A.2}
\begin{split}
\IV_{u,v} = & \{\sigma = (s_3, \dots, s_n) \in \IR_+^\Sigma; \;\mbox{$u - \mu_1 v \ge \ell_1 (\sigma)$ and $u - \mu_2 v \le \ell_2(\sigma)\}$, where}
\\
\ell(\sigma)  = &\Sigma^n_3 \Big(1 - \mbox{\f $\dis\frac{\mu_1}{\mu_i}$}\Big) \, s_i, \; \ell_2(\sigma) = \Sigma^n_3 \, \Big(1 - \mbox{\f $\dis\frac{\mu_2}{\mu_i}$}\Big) \, s_i, \; \mbox{for $\sigma \in \IR^\Sigma$}.
\end{split}
\end{equation}
For $\sigma \in \IV_{u,v}$ we define $s_1,s_2$ via
\begin{equation}\label{A.3}
\begin{split}
\mu_2(1 - 1/\mu_2) \, s_1 & = -u + \mu_2 v + \ell_2(\sigma),
\\
(1 - 1/\mu_2) \, s_2 & =\;\; u - \mu_1 v - \ell_1(\sigma).
\end{split}
\end{equation}
They are non-negative and determined by the identities
\begin{equation}\label{A.4}
\Sigma^n_1 \, s_i = u \; \mbox{and} \; \Sigma^n_1 \; \mbox{\f $\dis\frac{1}{\mu_i}$} \; s_i = v.
\end{equation}
We write
\begin{equation}\label{A.5}
\mbox{$V(u,v)$ for the ($(n-2)$-dimensional) volume of $\IV_{u,v}$}
\end{equation}
(It is a continuous function on $C$, positive on $\stackrel{_\circ}{C}$, see (\ref{A.11}), (\ref{A.18}) of \cite{SzniWidm26a}). 

\medskip
We then define the probability kernel from $C$ to $\IR_+^\Sigma$:
\begin{equation}\label{A.6}
\begin{split}
\wh{\mu}_{u,v} (d\sigma)  &=  1_{\IV_{u,v}} d \sigma / V_{u,v}, \; \mbox{when $(u,v) \in \, \stackrel{_\circ}{C}$}
\\
&=  \mbox{the Dirac mass at $0\; (\in \IR_+^\Sigma)$, when $u = v \ge 0$},
\\
&=  \mbox{the Dirac mass at $(0,\dots,u)$, when $u = \lambda_N v \ge 0$},
\end{split}
\end{equation}
so that $\wh{\mu}_{u,v} (\IV_{u,v}) = 1$, for all $(u,v) \in C$.

\medskip
With $\wh{\mu}_{u,v}$ we can construct an important probability kernel from $C$ to $\IR^N$. Namely, for $(u,v) \in C$, we define
\begin{equation}\label{A.7}
\begin{array}{l}
\mbox{$\mu_{u,v}(dx)$ to be the image measure on $\IR^N$ of $\wh{\mu}_{u,v}(d \sigma) \otimes \prod\limits_{i = 1}^n \; \mbox{\f $\dis\frac{d \theta_i}{2 \pi}$}$ on $\IR_+^\Sigma \times [0,2 \pi)^n$}
\\
\mbox{under the map $(\sigma, \theta_1,\dots,\theta_n) \r (\sqrt{s}_1  \cos \theta_1,  \sqrt{s}_1 \sin \theta_1, \dots , \sqrt{s}_n \cos \theta_n, \sqrt{s} _n \sin \theta_n)$}
\end{array}
\end{equation}
(note that $s_1,s_2$ are $\wh{\mu}_{u,v}$-a.s.~non-negative, see below (\ref{A.3})).

\begin{center}
\psfrag{l5}{$e_5$}
\psfrag{l4}{$e_4$}
\psfrag{l3}{$e_3$}
\includegraphics[width=4cm]{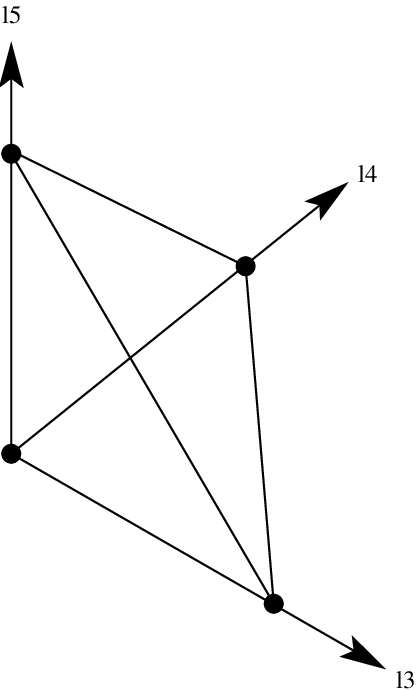} \qquad  \quad \includegraphics[width=4cm]{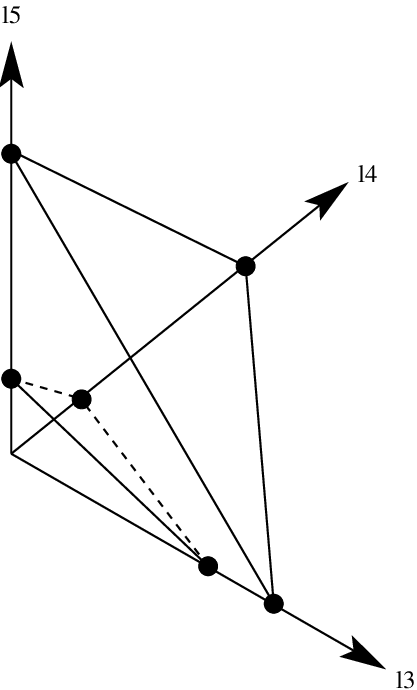} \qquad \quad  \includegraphics[width=4cm]{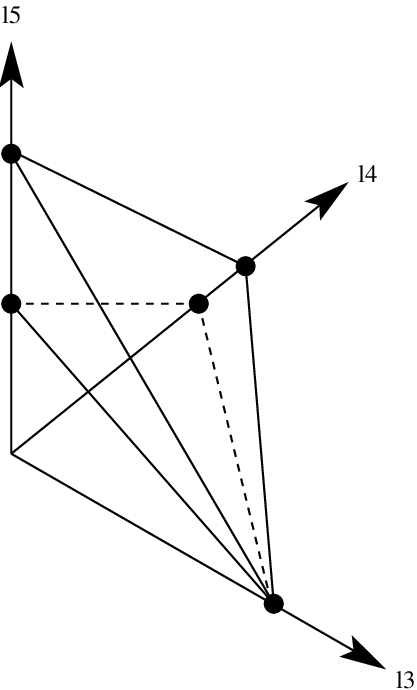}
\end{center}

\qquad $1 = \mu_1 < u/v \le \mu_2$ \qquad \qquad  \qquad $\mu_2 < u/v < \mu_3$  \qquad \qquad \qquad \quad $u/v = \mu_3$

\bigskip
\begin{center}
\psfrag{l5}{$e_5$}
\psfrag{l4}{$e_4$}
\psfrag{l3}{$e_3$}
\psfrag{s35}{$\sigma_{3,5}$}
\psfrag{s34}{$\sigma_{3,4}$}
\psfrag{s45}{$\sigma_{4,5}$}
\includegraphics[width=4cm]{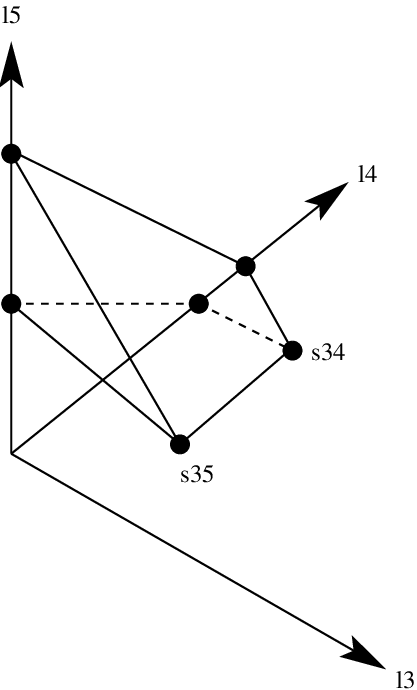} \qquad  \quad \includegraphics[width=4cm]{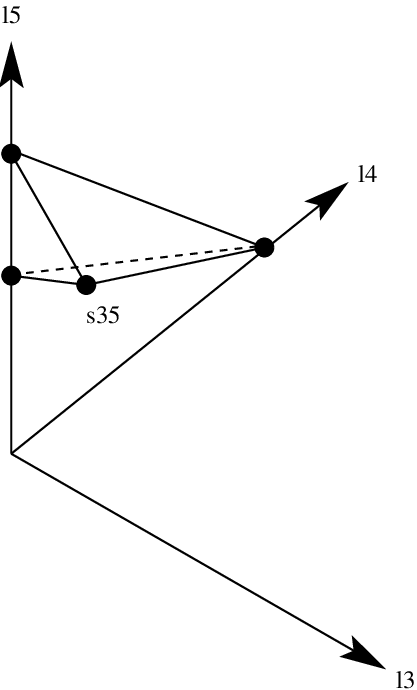} \qquad \quad  \includegraphics[width=4cm]{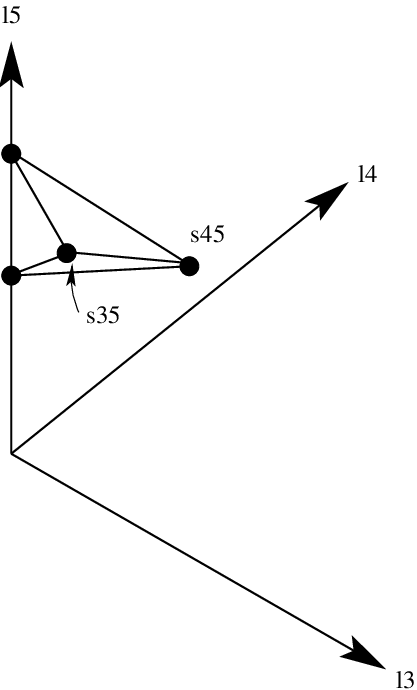}
\end{center}

\qquad $\mu_3 < u/v < \mu_4$ \qquad \qquad  \qquad \qquad  $u/v = \mu_4$  \qquad \qquad \qquad \qquad $\mu_4 < u/v < \mu_5$

\bigskip\bigskip\n
Fig.~1: A schematic illustration of the polytope $\IV_{u,v}$ when $n=5$, with varying ratio $u/v$.

\medskip
The next lemma highlights the interest of keeping the ratio $u/v$ away from the ``singular values'' $\mu_i$, and the special role of the kernel $\mu_{u,v}(dx)$.

\begin{lemma}\label{lemA.1}
If $(u,v) \in \, \stackrel{_\circ}{C}, i_0 \in \{1, \dots, n-1\}$ and $\eta > 0$ are such that
\begin{equation}\label{A.8}
\mu_{i_0} + \eta \le u/v \le \mu_{i_0 + 1} - \eta,
\end{equation}
then
\begin{equation}\label{A.9}
\begin{array}{l}
\mbox{for any $x \in \IX_{u,v}$ there exists $i_0 \in \{1, \dots , n-1\}$, $a \le 2 i_0$, $b \ge 2i_0 + 1$ in $\{1,\dots, N\}$}
\\
\mbox{such that $\min\{x^2_a, x^2_b\} \ge c_0(\eta) \, u$}.
\end{array}
\end{equation}
For $(u,v) \in  \, \stackrel{_\circ}{C}$ with $u/v \not= \mu_i$, for $1 \le i \le n$,
\begin{equation}\label{A.10}
\mbox{$\IX_{u,v}$ is a smooth, compact, connected submanifold of $\IR^N$}.
\end{equation}
\begin{equation}\label{A.11}
\mbox{For $(u,v) \in C, \mu_{u,v} (\IX_{u,v}) = 1$}.
\end{equation}
Moreover,
\begin{equation}\label{A.12}
\begin{array}{l}
\mbox{$\mu_{u,v}(dx)$ is a regular conditional probability for the law $\mu(dx)$ on $\IR^N$ given}
\\
\mbox{$|x|^2 = u$ and $|x|^2_{-1} = v$, with $(u,v) \in C$. It maps, as a kernel, continuous}
\\
\mbox{bounded functions on $\IR^N$ into continuous bounded functions on $C$}.
\end{array}
\end{equation}
(Incidentally, $\mu_{u,v}$ does not depend on the parameter $a$ in {\rm (\ref{1.5})}, but $\mu$ in {\rm (\ref{1.6})} does).
\end{lemma}

\begin{proof}
We begin with the proof of (\ref{A.9}) when assuming (\ref{A.8}). Writing $s_i = x^2_{2i-1} + x^2_{2i}$ for $1 \le i \le n$, we have by the first inequality of (\ref{A.8})
\begin{equation}\label{A.13}
u = s_1 + \dots + s_n \ge (\mu_{i_0} + \eta) (s_1 + s_2/ \mu_2 + \dots + s_n/\mu_n).
\end{equation}
Setting $j_0 = i_0 + 1$, we find that
\begin{equation}\label{A.14}
\begin{array}{l}
s_{j_0} \big(1 - (\mu_{i_0} + \eta)/ \mu_{j_0}\big) + \dots + s_n \big(1- (\mu_{i_0} + \eta)/ \mu_n\big) \ge
\\ 
(\mu_{i_0} + \eta-1)s_1 + ((\mu_{i_0} + \eta)/\mu_2 -1) s_2+ \dots + ((\mu_{i_0} + \eta)/\mu_{i_0} -1 ) s_{i_0},
\end{array}
\end{equation}
and hence by (\ref{A.8})
\begin{equation}\label{A.15}
s_{j_0} + \dots + s_n \ge c\,(\eta) (s_1 + \dots + s_n) = c\,(\eta) \,u.
\end{equation}
Likewise, by the second inequality of (\ref{A.8}), we also have (recall that $j_0 = i_0 + 1$) $s_1 + \dots + s_n \le (\mu_{j_0} - \eta) (s_1 + s_2/\mu_2 + \dots + s_n/\mu_n)$ so that
\begin{equation}\label{A.16}
\begin{array}{l}
s_{j_0} \big(1 - (\mu_{j_0} - \eta)/\mu_{j_0}\big) + \dots + s_n \big(1- (\mu_{j_0} - \eta)/ \mu_n)\big) \le 
\\
s_1 (\mu_{j_0} - \eta - 1) \;+ s_2 \big((\mu_{j_0} - \eta) / \mu_2 - 1\big) + \dots + s_{i_0} \big((\mu_{j_0} - \eta) / \mu_{i_0} - 1\big)\big)
\end{array}
\end{equation}
and hence (\ref{A.8})
\begin{equation}\label{A.17}
s_1 + \dots + s_{i_0} \ge c\,(\eta) (s_1 + \dots + s_n) = c\, (\eta) \, u.
\end{equation}
Combining (\ref{A.15}) and (\ref{A.17}), and recalling the definition of $s_i$ above (\ref{A.13}), the claim (\ref{A.9}) follows.

\medskip
We now turn to the proof of (\ref{A.10}). Since $u/v \not= \mu_i$ for $1 \le i \le n$, we know by (\ref{A.9}) that for any $x$ in $\IX_{u,v}$ we can find two components $x_\ell, x_{\ell^\prime}$ different from zero and with $\lambda_\ell \not= \lambda_{\ell^\prime}$. Hence, $x$ and $\Lambda^{-1} x$ are not colinear (see below (\ref{1.13}) for the definition of $\Lambda$), and the differential at $x$ of the $\IR^2$-valued map $(| \cdot |^2, | \cdot |^2_{-1})$ has full rank. It follows that $\IX_{u,v}$ is a smooth compact submanifold of $\IR^N$ (see for instance \cite{Warn83} on p.~31). There remains to prove that it is connected. To this effect we set for $x \in \IR^N$,
\begin{equation}\label{A.18}
s_i(x) = x^2_{2i-1} + x^2_{2i}, 1 \le i \le n, \; \mbox{and} \; \sigma(x) = \big(s_i(x)\big)_{3 \le i \le n} \in \IR_+^\Sigma,
\end{equation}
so that
\begin{equation}\label{A.19}
\begin{split}
\IX_{u,v} = \big\{x \in \IR^N; &\; \sigma(x) \in \IV_{u,v}, \mu_2(1 - 1/\mu_2) \,s_1(x) = -u + \mu_2 v + \ell_2 \big(\sigma(x)\big),
\\
&\;\mbox{and} \; (1 - 1/\mu_2) \, s_2(x) = u - \mu_1 v - \ell_1\big(\sigma(x)\big)\big\}.
\end{split}
\end{equation}
Now, any two points of the polytope $\IV_{u,v}$ can be joined by a linear segment within $\IV_{u,v}$, and the inverse image in $\IR^N$ under $\sigma$ of a point in $\IV_{u,v}$ is a product of $n$ (possibly degenerate) circles. One can then ``lift a linear path'' linking the images under $\sigma$ of two points in $\IX_{u,v}$ and construct a continuous path in $\IX_{u,v}$ between these two points. The claim (\ref{A.10}) follows.

\medskip
The claim (\ref{A.11}) readily follows from (\ref{A.6}) and (\ref{A.4}).

\medskip
As for (\ref{A.12}), we first note that when $S_1,\dots,S_n, \Theta_1, \dots, \Theta_n$ are independent random variables with the $S_i, 1 \le i \le n$, having exponential distribution with parameter $1/a$, and the $\Theta_i, 1 \le i \le n$, being uniformly distributed over $[0, 2 \pi)$, then the random vector $(S_1^{1/2} \cos \Theta_1, S_1^{1/2} \sin \Theta_1, \dots, S^{1/2}_n \cos \Theta_n$, $S_n^{1/2} \sin \Theta_n)$ has distribution $\mu$. If $\Phi$ is a bounded measurable function on $R^{N+2}$, setting $U = S_1 + \dots + S_n$, $V = S_1 + S_2 / \mu_2 + \dots + S_n / \mu_n$, one has with hopefully obvious notation, and a similar change of variables as in (\ref{A.3}) of \cite{SzniWidm26a}:
\begin{equation}\label{A.20}
\begin{array}{l}
E^\mu[\Phi(|x|^2, |x|^2_{-1}, x_1,x_2, \dots,x_{N-1},x_N)] = 
\\[1ex]
E[\Phi(U, V, S_1^{1/2} \cos \Theta_1, S_1^{1/2} \sin \Theta_1, \dots , S_n^{1/2} \cos \Theta_n, S^{1/2}_n \sin \Theta_n)] = 
\\[0.5ex]
\dis\int \Phi (u,v, \sqrt{s}_1\, \cos \theta_1, \sqrt{s}_1 \sin \theta_1, \dots, \sqrt{s}_n \, \cos \theta_n, \sqrt{s}_n \sin \theta_n) \; 1\{(u,v) \in 
\\[0.5ex]
C,(s_3, \dots, s_n) \in \IV_{u,v}, \theta_i \in [0, 2 \pi), 1 \le i \le n\} 
\\[0.5ex]
 \mbox{\f $\dis\frac{1}{(2 \pi)^n}$} \; \mbox{\f $\dis\frac{e^{-u/a}}{a^n(1 - \mu_2^{-1})}$} \;du \,dv\,ds_3 \dots ds_n \ d \theta_1 \dots d \theta_n \stackrel{(\ref{A.7})}{=}
\\[2ex]
E^\mu \Big[ \dis\int \Phi (|x|^2, |x|^2_{-1}, z) \, d\mu_{|x|^2,|x|^2_{-1}}(z)\Big],
\end{array}
\end{equation}
and combined with (\ref{A.11}), $\mu_{u,v}(dz)$ is a regular conditional probability for the law $\mu$ given $|x|^2 = u$ and $|x|^2_{-1}  = v$. As for the fact it maps continuous bounded functions on $\IR^N$ into continuous bounded functions on $C$, this readily follows from (\ref{A.7}) and a similar property for $\wh{\mu}_{u,v}$. This last point follows from the fact that $(u,v) \in \, \stackrel{_\circ}{C} \;\r 1_{\IV_{u,v}}(\sigma) \,d \sigma \in L^1(d\sigma)$ is continuous, see for instance (\ref{A.11}) of \cite{SzniWidm26a}, and from the fact that when $(u_m,v_m)$ in $C$ converges to $(u_*,v_*) \in \partial C$, then the Hausdorff distance of $\IV_{u_m, v_m}$ to $0 \in \IR^\Sigma$, when $u_* = v_*$, or to $(0, \dots,u_*) \in \IR^\Sigma$, when $u_* = \lambda_N v_*$, tends to zero with $m$, and from the continuity of the functions $s_1,s_2$ in (\ref{A.3}) (extended to $C \times \IR^\Sigma$).

\medskip
This completes the proof of (\ref{A.12}) and hence of Lemma \ref{lemA.1}. 
\end{proof}

\begin{remark}\label{remA.2} \rm
We highlight the following significant consequence of (\ref{A.12}) and (\ref{2.9}), which in particular is used in Section 5, see  (\ref{5.25}):
\begin{equation}\label{A.21}
q_\ell (u,v) = \dis\int x^2_\ell \, d\mu_{u,v} (x), \; \mbox{for $(u,v) \in C$ and $1 \le \ell \le N$}.
\end{equation}
Indeed, replacing $u,v$ by $|x|^2$ and $|x|^2_{-1}$ in both members, one obtains $E^\mu [x^2_\ell \, | \, |x|^2, |x|^2_{-1}]$ by (\ref{2.9}) and (\ref{A.12}). This implies the equality of the two members of (\ref{2.9}) for a.e.~$(u,v)$ in $C$, and hence for all $(u,v)$ in $C$ by continuity. \hfill $\square$
\end{remark}

Next, we will show that the collection of vector fields $Z_m, 1 \le m \le M$, see (\ref{1.16}), is rich enough to ensure an efficient stirring on $\IX_{u,v}$, when $(u,v)$ in $C$ stays away from the singular set where $u = \mu_i v$ for some $i$. We refer to (\ref{1.14}), (\ref{1.15}) for notation.

\begin{proposition}\label{propA.3} {\it (full rank property)}

\medskip\n
Given $\eta > 0$ and $(u,v)$ in $\stackrel{_\circ}{C}$ such that
\begin{equation}\label{A.22}
\mbox{$u/v$ is at distance at least $\eta$ from each $\mu_i, 1 \le i \le n$},
\end{equation}
then, for $x \in \IX_{u,v}$ and $z$ in $\IR^N$ with $\langle z,x\rangle = 0 = \langle z,x\rangle_{-1}$, one has
\begin{equation}\label{A.23}
\Sigma_{J \in \cJ} \; \mbox{\f $\dis\frac{1}{|x|^4}$} \; \langle T_J (x), z \rangle^2 + \Sigma^n_1 \;  \mbox{\f $\dis\frac{1}{|x|^2}$} \; \langle R_i(x), z\rangle^2 \ge c_1(\eta) \, |z|^2.
\end{equation}
\end{proposition}

\begin{proof}
We begin with some reductions. By homogeneity, see (\ref{1.14}), (\ref{1.15}), we can assume that $|x|^2 = u = 1$ and $|z|^2 = 1$. With the help of (\ref{A.8}), (\ref{A.9}) we can also assume that $i_0$ in $\{1, \dots, n-1\}, a \le 2i_0, b \ge 2 i_0 + 1$ are fixed and that $\min \{x^2_a,x^2_b\} \ge c_0(\eta)$.

\medskip
We now choose a sequence $\ell_1,\dots,\ell_{N-4}$ of distinct integers in $\{1, \dots , N\}$ such that the values $\lambda_{\ell_1}, \lambda_{\ell_2}, \dots, \lambda_{\ell_{N-4}}$ are all different from $\lambda_a$ and $\lambda_b$. We denote by $J_1, J_2, \dots, J_{N-4}$ the triples that are the increasing orderings of the sets $\{a,b,\ell_1\}, \{a,b,\ell_2\}, \dots , \{a,b,\ell_{N-4}\}$. Letting $T_{J_1}(x)$ act by scalar product (i.e. $\langle \cdot, \cdot \rangle$) on vectors in $\IR^N$ with vanishing $a,b$ components, we see that ${\rm Span} \{T_J(x)\}$ has dimension $1$. Then, by letting $T_{J_1}(x), T_{J_2}(x)$ act on vectors in $\IR^N$ with vanishing $a,b,\ell_1$ components, we see that $T_{J_2}(x)$ is not colinear to $T_{J_1}(x)$ and ${\rm Span}\{T_{J_1}(x),T_{J_2}(x)\}$ has dimension 2. And by induction, letting $T_{J_1}(x), \dots, T_{J_{k+1}}(x)$ act on vectors with vanishing $a,b,\ell_1,\dots,\ell_k$-components, we see that ${\rm Span}\{T_{J_1}(x), \dots, T_{J_{k+1}}(x)\}$ has dimension $k + 1$ as long as $k + 1 \le N-4$. Adding successively $R_a(x)$ and $R_b(x)$ to this collection further increases the dimension by $2$ (using a similar argument). As a result we find that ${\rm Span}\{T_J(x), \dots, T_{J_{N-4}}(x)$, $R_a(x), R_b(x)\}$ which is contained in $\{y \in \IR^N; \langle y,x\rangle = 0 = \langle y,x\rangle_{-1}\}$ by (\ref{1.17}) has dimension $N-2$ and hence coincides with $\{y \in \IR^N; \langle y, x\rangle = 0 = \langle y,x\rangle_{-1}\}$. This shows that the left member of (\ref{A.23}) is positive when $x,z$ are such that $|x|^2 = 1, |x|^2 / |x|^2_{-1}$ is at distance at least $\eta$ from each $\mu_i, 1 \le i \le n$, $|z|^2 = 1, \langle x, z\rangle = 0 = \langle x,z\rangle_{-1}$. This is a compact set, and the left member of (\ref{A.22}) is a continuous positive function on this compact set, and thus admits a positive lower bound $c_1(\eta)$. The claim (\ref{A.23}) follows. 
\end{proof}

We conclude this appendix with a result concerning local charts, which will be helpful when proving the convergence to equilibrium of the diffusion on $\IX_{u,v}$ induced by drift and stirring in Appendix B.

\begin{lemma}\label{lemA.4} (local charts)

\medskip\n
Given $i_0 \in \{2,\dots,n\}$, $(u^0, v^0) \in \Delta_{i_0}$ (see {\rm (\ref{2.5})} for notation), $x^0 \in \IX_{u^0,v^0}, a, b$ in $\{1,\dots,N\}$ with $\lambda_a \not= \lambda_b$ and $x^0_a \, x^0_b \not= 0$, one can find an open Euclidean ball $B^0$ centered at $(x^0_\ell)_{\ell \not= a,b}$ in $\IR^{\{1,\dots,N\} \backslash \{a,b\}}$, an open rectangle $R^0$ centered at $(u^0,v^0)$ with closure contained in $\Delta_{i_0}$, a rectangle $Q^0$ in $R^{\{a,b\}}$, product of two open intervals bounded away from zero and respectively centered at $x^0_a$ and $x^0_b$, such that 
\begin{equation}\label{A.24}
\begin{array}{l}
\mbox{for $(x_\ell)_{\ell \not= a,b}$ in $B^0, (x_a,x_b)$ in $Q^0$, $(u,v)$ in $R^0$, the set where $x \in \IX_{u,v}$}
\\
\mbox{coincides with the graph of a smooth $Q^0$-valued function $F$ on $B^0 \times R^0$.}
\\
\mbox{And, for $(u,v)$ in $R^0$, $\mu_{u,v}$ restricted to the set where $(x_\ell)_{\ell \not= a,b} \in B^0$,}
\\
\mbox{$(x_a,x_b) \in Q^0$, is bounded above and below by constant multiples of the}
\\
\mbox{image of the Lebesgue measure on $B^0$ under the ``graph map'' of $F(\cdot,(u,v))$.}
\end{array}
\end{equation}
\end{lemma}

\begin{proof}
We view the $\IR^2$-valued smooth function $(|x|^2 - u, |x|^2_{-1} - v)$ on $\IR^{N+2}$ as a function $\varphi(X,Y)$ of the variables $X = ((x_\ell)_{\ell \not= a,b},u,v) \in \IR^N$ and $Y = (x_a,x_b) \in \IR^2$. The determinant of $\frac{\partial \varphi}{\partial Y}$ equals $2 (\frac{1}{\lambda_a} - \frac{1}{\lambda_b})\, x_a \, x_b$ and does not vanish near $(X^0,Y^0)$ (corresponding to $x^0, (u^0,v^0)$). The first part of the claim follows from the application of the Implicit Function Theorem, see Th.~9.15-1 on p.~818 of \cite{Ciar25}. For the second part, one uses the change of variable formula, and the approximation for $(u,v)$ in $R^0$ of $\mu_{u,v}$ restricted to the set where $(x_\ell)_{\ell \not= a,b} \in B^0$, $(x_a,x_b) \in Q^0$ by the conditional measure $\mu(\,\cdot\,\, |\,|\, |x|^2 - u| < \rho, | \, |x|^2_{-1} - v| < \rho)$ restricted to the set $(x_\ell)_{\ell \not= a,b} \in B^0$, $(x_a,x_b) \in Q^0$, as $\rho \r 0$. (Incidentally, one can set the parameter $a$ in (\ref{1.5}) governing the variance of $\mu$ equal to $1$ for this step, so that $a$ does not enter the constants in the statement of Lemma \ref{lemA.4}). 
\end{proof}

\section{Appendix: Convergence to equilibrium on $\IX_{u,v}$}
\setcounter{equation}{0}

In this appendix the main objective is to establish a quantitative convergence to equilibrium on $\IX_{u,v}$ for the semigroup generated by $B + \kappa \cD$, see (\ref{1.10}), (\ref{1.19}), (\ref{1.20}), when $(u,v)$ remains away from the singular rays where $u = \mu_i v$, as well as bounded and bounded away from $0$, see Proposition \ref{propB.2}. Throughout this appendix positive constants will tacitly depend on $N, \lambda_\ell, 1 \le \ell \le N$, the coefficients entering the definition of $B$ in (\ref{1.10}), and $\kappa$ in (\ref{1.20}).

\medskip
In the next lemma $Z$ stands for one of the vector fields in (\ref{1.10}), (\ref{1.16}).

\begin{lemma}\label{lemB.1}
For $\wt{f}, \wt{g}$ smooth functions on $\IR^N$ growing at most polynomially together with their derivatives and $Z$ as above, one has
\begin{equation}\label{B.1}
\dis\int Z \wt{f} \; \wt{g} \, d\mu = - \dis \int \, \wt{f} \; Z\, \wt{g} \; d\mu .
\end{equation}
If $(u,v) \in \, \stackrel{_\circ}{C}$ and $u/v \not= \mu_i$, for $1 \le i \le n$, then for $f,g$ smooth functions on $\IX_{u,v}$ (see {\rm \ref{A.10})}) and $Z$ as above (and thus tangent vector field to $\IX_{u,v})$, one has
\begin{equation}\label{B.2}
\dis\int Z f \; g \, d\mu_{u,v} = - \dis \int \, f \; Z g \; d\mu_{u,v} .
\end{equation}
\end{lemma}

\begin{proof}
The identity (\ref{B.1}) follows directly from the fact that the vector fields $Z$ in (\ref{1.10}), (\ref{1.16}), satisfy (\ref{1.11}), (\ref{1.17}), and from integration by parts. We then turn to (\ref{B.2}). By (\ref{A.10}), $\IX_{u,v}$ is a smooth compact submanifold of $\IR^N$, and we can extend $f,g$ as smooth compactly supported functions $\wt{f}, \wt{g}$ on $\IR^N$ (see for instance \cite{Warn83}, p.~29). If $\psi$ is a smooth non-negative compactly supported function on $\stackrel{_\circ}{C}$ with $\int \psi (|x|^2, |x|^2_{-1}) \,d\mu = 1$, we find by the application of (\ref{B.1}) to $\wt{f}$ and $\wt{g} \, h$, with $h(x) = \psi (|x|^2, |x|^2_{-1})$, that
\begin{equation}\label{B.3}
\dis\int Z \wt{f} \; \,\wt{g} \, h \, d \mu = - \dis\int \wt{f} \, Z(\wt{g} \, h) \, d \mu \stackrel{(\ref{1.18})}{=} - \dis\int \wt{f} \, Z \wt{g} \, \, h \, d\mu .
\end{equation}
Using (\ref{A.12}) and letting $\psi$ concentrate near $(u,v)$, the identity (\ref{B.2}) follows since $Z \wt{f}$, $Z \wt{g}$ respectively coincide with $Z f$ and $Zg$ on $\IX_{u,v}$. 
\end{proof}

We will now describe the diffusion process on $\IX_{u,v}$ attached to $B + \kappa \cD$ by means of the law of the solution of a Stratonovich differential (with quadratic and linear coefficients) on $\IX_{u,v}$ (see Chapter 5 of \cite{IkedWata89} for background material). With this in mind, we consider
\begin{equation}\label{B.4}
\mbox{$(u,v)$ in $\stackrel{_\circ}{C}$ with $u/v \not= \mu_i$, for all $1 \le i \le n$,}
\end{equation}
and the solution of the Stratonovich differential equation
\begin{equation}\label{B.5}
\begin{split}
d Y_t & = \sqrt{\kappa} \; \dis\sum\limits^M_{m=1} Z_m (Y_t) \circ d \beta_m (t) + b(Y_t) \, dt 
\\
Y_0 & = y \in \IX_{u,v},
\end{split}
\end{equation}
with $0 < \kappa \le 1, Z_m, 1 \le m \le M, b$ as in (\ref{1.20}), (\ref{1.16}), (\ref{1.10}), and $\beta_m, 1 \le m \le M$, independent Brownian motions. Due to (\ref{1.17}), (\ref{1.11}), the diffusion $(Y_t)_{t \ge 0}$ lives on the smooth compact manifold $\IX_{u,v}$, and its associated generator is the operator
\begin{equation}\label{B.6}
\Gamma = \kappa \cD + B.
\end{equation}
It is elliptic by (\ref{A.23}). To use duality, as in (\ref{B.2}), it is convenient to introduce the diffusion $(Y^*_t)_{t \ge 0}$ solution of the Stratonovich differential equation:
\begin{equation}\label{B.7}
\begin{split}
d Y_t^* & = \sqrt{\kappa} \; \dis\sum\limits^M_{m=1} Z_m (Y_t^*) \circ d \beta_m (t) - b(Y_t^*) \, dt 
\\
Y^*_0 & = y^* \in \IX_{u,v},
\end{split}
\end{equation}
which also lives on $\IX_{u,v}$ and has the associated generator 
\begin{equation}\label{B.8}
\Gamma^* = \kappa \cD - B.
\end{equation}
With the help of (\ref{B.2}), one has for all smooth functions $f,g$ in $\IX_{u,v}$:
\begin{equation}\label{B.9}
\dis\int \Gamma f \; g \; d\mu_{u,v} = \dis\int f \; \Gamma^* g \; d\mu_{u,v}.
\end{equation}
The operators $\Gamma$ and $\Gamma^*$ are smooth and uniformly elliptic on $\IX_{u,v}$ due to the full rank property (\ref{A.23}). Moreover, the identity (\ref{B.9}) extends to the semigroups (by performing a differentiation in $s$ of $\int e^{(t-s)\Gamma} \! f \, \,e^{s \Gamma^*} g \, \, d \mu_{u,v})$:
\begin{equation}\label{B.10}
\dis\int e^{t \Gamma}\! f \; g \, d\mu_{u,v} = \dis\int f \; e^{t \Gamma^*} \! g \, d\mu_{u,v}, \; \mbox{for $t \ge 0$}.
\end{equation}
Further, one has for $t > 0$ the transition densities:
\begin{equation}\label{B.11}
\begin{array}{l}
e^{t \Gamma} \!f(x) = \dis\int p_t (x,z)\, f(z) \,d\mu_{u,v}(z), \; e^{t \Gamma^*} \!g(y) = \dis\int p^*_t (y,z) g(z) \, d\mu_{u,v}(z)
\\
\mbox{with $p_t(x,y) = p^*_t(y,x)$ smooth in $x$ and in $y$}
\end{array}
\end{equation}
(see also Chapter 5 \S 2 and \S 3 of \cite{IkedWata89}).

\medskip
Note that by (\ref{B.10}) with $g = 1$, $\mu_{u,v}$ is a stationary distribution of $e^{t \Gamma}, t \ge 0$ (and choosing $f = 1$, it is a stationary distribution of $e^{t \Gamma^*}, t \ge 0$ as well). We now come to the main purpose of this appendix (see (\ref{B.11}) for notation, and we refer to the beginning of this appendix for our convention concerning constants).

\begin{proposition}\label{propB.2}
Consider
\begin{equation}\label{B.12}
0 < u_{\min} < u_{\max},  \,\,\eta > 0,
\end{equation}
and suppose that $(u,v) \in C$ satisfies
\begin{equation}\label{B.13}
\mbox{$u_{\min} \le u \le u_{\max}$ and $u/v$ is at distance at least $\eta$ from each $\mu_i, 1 \le i \le n$}.
\end{equation}
Then, one has positive constants $c_2(u_{\min}, u_{\max}, \eta)$ and $c_3(u_{\min}, u_{\max}, \eta)$, such that
\begin{equation}\label{B.14}
|p_t (x,y) - 1| \le c_2 \, \exp\{-c_3 \,t\}, \; \mbox{for all $t \ge 1$ and $x,y$ in $\IX_{u,v}$}.
\end{equation}
\end{proposition}

\begin{proof}
Consider $(u,v)$ as above and set for $\tau \ge 0$
\begin{equation}\label{B.15}
\ov{d} (\tau) = \sup\limits_{x,x^\prime \in \IX_{u,v}} \; \fr \; \dis\int  |p_\tau (x,z) - p_\tau (x^\prime, z) \, | \, d\mu_{u,v} (z)
\end{equation}
(understood as $1$ when $\tau = 0$).

\medskip
The quantity under the supremum is the so-called {\it total variation distance} of the probabilities $p_\tau(x,z)$ $d\mu_{u,v}(z)$ and $p_\tau(x^\prime,z) \, d\mu_{u,v}(z)$. The function $\tau \ge 0 \r \ov{d} (\tau)$ is known to be non-increasing and sub-multiplicative  (i.e.~$\ov{d} (\tau_1 + \tau_2) \le \ov{d}(\tau_1) \, \ov{d}(\tau_2)$), see for instance Lemma 4.12 on p.~54 of \cite{LeviPereWilm17}. As we now explain, the claim (\ref{B.14}) will follow once we show that for $u,v,\eta$, as in (\ref{B.12}), (\ref{B.13}) one has
\begin{align}
p_1(x,y) & \ge c(u_{\min},u_{\max},\eta) > 0, \; \mbox{for $x,y \in \IX_{u,v}$, and} \label{B.16}
\\[1ex]
p_1(x,y) & \le c(u_{\min},u_{\max},\eta) > 0, \; \mbox{for $x,y \in \IX_{u,v}$}. \label{B.17}
\end{align}
Indeed (and letting from now on all positive constants implicitly depend on $u_{\min},u_{\max}, \eta$ to alleviate notation), it follows from (\ref{B.16}) that $\ov{d}(1) \le e^{-c}$, so that for $t \ge 1$ and $x,y$ in $\IX_{u,v}$ (and an obvious Dirac mass meaning when $t = 1$), one has by the stationarity of $\mu_{u,v}$ and the semigroup property:
\begin{equation}\label{B.18}
\begin{array}{l}
\frac{1}{2}\; |p_t(x,y) -1 | = \Big| {\dis\int} \frac{1}{2} \big(p_{t - 1} (x,z) - p_{t-1} (x^\prime,z)\big) \, p_1(z,y) \, d \mu_{u,v}(z) \, d\mu_{u,v}(x^\prime)\Big| \le
\\
\ov{d} (t-1) \sup \, p_1(\cdot, \cdot) \le \exp \{ - c[t-1]\} \, c \le c^\prime \exp\{-c t\},
\end{array}
\end{equation}
and this will prove (\ref{B.14}).

\medskip
We thus need to show (\ref{B.16}) and (\ref{B.17}). For this purpose we use the local charts constructed in Lemma \ref{lemA.4}. The set of $x$ in $\IX_{u,v}$ and $(u,v)$ as in (\ref{B.13}) is a compact subset of $\IR^N \times (0, \infty)^2$ (it is a closed bounded subset of $\IR^N \times \IR^2$ which is contained in $\IR^N \times (0, \infty)^2$). At each point $(x^0,u^0,v^0)$ of this compact subset we have $i_0$ in $\{2,\dots,n\}$ such that $(u^0,v^0) \in \Delta_{i_0}$, and by Lemma \ref{lemA.1}, $a^0, b^0$ in $\{1,\dots, N\}$ with $\lambda_{a^0} \not= \lambda_{b^0}$ such that $x^0_{a^0} \, x^0_{b^0} \not= 0$. So we are in the set-up where Lemma \ref{lemA.4} applies. We use compactness to extract a finite collection of cardinality $K$ of points $(x^0,u^0,v^0)$ with corresponding $i_0, a^0,b^0$, together with concentric non-degenerate balls $B^0_1 \subseteq B^0_2 \subseteq B^0_3 \subseteq B^0_4$ of $\IR^{N-2}$ centered at $(x^0_\ell)_{\ell \not= a^0,b^0}$ (with triple radius of the preceding one for $B^0_4, B^0_3,B^0_2$) with $B^0_1, B^0_2, B^0_3$ closed, $B^0_4$ open, and $\ov{B}_4\,^{\!\!\!\!0}$ contained in the concentric (open) ball $B^0$ of Lemma \ref{lemA.4}, as well as rectangles $Q^0, R^0 \supseteq R^0_1$, respectively centered at $(x^0_{a^0}, x^0_{b^0})$ and $(u^0,v^0)$ with $Q^0, R^0$ open, and $R^0_1$ closed, concentric to $R^0$ and of half-size, and a smooth $Q^0$-valued function $F^0$ on $B^0 \times R^0$ such that (\ref{A.24}) holds, and the graphs of the functions $F^0$ restricted to $B^0_1 \times R^0_1$ of this finite collection covers the above compact set of $x$ in $\IX_{u,v}$ and $(u,v)$ as in (\ref{B.13}).

\medskip
For each $u_{\min}, u_{\max}, \eta$ as in (\ref{B.13}), we view the above choices as fixed. For convenience, we refer to it as the ``charts'' associated to $u_{\min},u_{\max}, \eta$ (with $K(u_{\min}, u_{\max}, \eta)$ the number of these charts), and this fixed choice of charts may enter constants below.

\medskip
Given such a chart and $(u,v) \in R^0_1$, we write
\begin{equation}\label{B.19}
\mbox{$C_1^{u,v} \subseteq C_2^{u,v} \subseteq C_3^{u,v} \subseteq O^{u,v}$ for the images of $B^0_1, B^0_2, B^0_3, B^0_4$ under $(\cdot, F^0\big( \cdot, u,v)\big)$}
\end{equation}
(so $C_1^{u,v}, C_2^{u,v}, C_3^{u,v}$ are closed and $O^{u,v}$ open in $\IX_{u,v}$), and denote by
\begin{equation}\label{B.20}
p_t^{O^{u,v}} (x,y), \; t \ge 0, \, x,y \in \IX_{u,v},
\end{equation}

\n
the Dirichlet transition densities of the diffusion in (\ref{B.5}) (with generator $\Gamma$) killed outside $O^{u,v}$ (it vanishes if $x$ or $y$ lies outside $O^{u,v}$). To alleviate notation, we tacitly omit $u,v$ from the notation in what follows.

\medskip
We now collect some useful upper and lower bounds on the above Dirichlet heat kernels. Namely, one has (with $K$ as above (\ref{B.19}))
\begin{align}
\inf \{p^O_t (x,y); & \; 1 \ge t \ge 1/K, x \in C_3, y \in C_3\} \ge c \; \;\mbox{and} \label{B.21}
\\[1ex]
\sup \{p^O_1 (x,y); & \; x \in C_2, y \in C_2\}  \vee \sup\{p_s^O (x,y), 0 < s \le 1, x \in \partial C_2, y \in C_1\} \le c^\prime  \label{B.22}
\end{align}

\n
(and $\partial C_2$ refers to the relative boundary of $C_2^{u,v} \subseteq O^{u,v}$).

\medskip
To prove these estimates, one uses the local chart to express $p_s^O(x,y)$ in terms of the fundamental solution of a second order parabolic equation, which is uniformly elliptic thanks to Proposition \ref{propA.3}, and has uniformly Lipschitz coefficients, with Dirichlet boundary conditions on $\partial B^0_4$ ($\subseteq \IR^{N-2}$). The above mentioned uniform ellipticity and Lipschitz property actually holds in a uniform neighborhood of $B^0_3$ and of $R^0_1$ (for $(u,v)$), see above (\ref{B.19}). One can thus extend the diffusion operators to the full space $\IR^{N-2}$ in a uniformly elliptic and Lipschitz fashion. Using Duhamel's formula, one can derive small time lower bound on the Dirichlet heat kernel for nearby points in a compact neighborhood of $B^0_3$ contained in $B^0_4$, with the help of the full space heat kernel estimates (see Theorem 1 on p.~67 and (4.75) on p.~82 of \cite{IlinKalaOlei62}), and then a chaining argument, to deduce (\ref{B.21}). The estimates (\ref{B.22}) are simpler to obtain; they rely on the upper bound of the full space heat kernel estimates in Theorem 1 of the above quoted reference. This proves (\ref{B.21}) and (\ref{B.22}).

\medskip
To proceed, we now introduce the positive constant (see above (\ref{B.19}))
\begin{equation}\label{B.23}
\gamma = \inf\{\mu^{u,v} (C^{u,v}_3 \cap \wt{C}^{u,v}_1); \; (u,v) \in R^0_1 \cap \wt{R}^0_1 \not= 0, \; C_2^{u,v} \cap \wt{C}^{u,v}_1 \not= \emptyset\} \in (0,1]
\end{equation}
(where the infimum refers to different charts, with the above constraints).

\medskip
As we now explain, for $(u,v)$ as in (\ref{B.13}), one has
\begin{equation}\label{B.24}
\inf \{p_1(x,y); x,y \in \IX_{u,v}\} \ge \inf\{p^O_t(z, C); \; 1 \ge t \ge 1/K, z \in C_3, \wt{z} \in C_3\}^K \, \gamma^{K-1}
\end{equation}
(the infimum in the right member runs over all charts).

\medskip
Combined with (\ref{B.21}), the lower bound (\ref{B.24}) readily implies (\ref{B.16}).

\medskip
There remains to prove (\ref{B.24}). We consider $(u,v)$ as in (\ref{B.13}) and $x,y \in \IX_{u,v}$. We will construct a sequence of distinct charts indexed with $0 \le j  < K_0$, where $K_0$ is at most $K$, such that (with hopefully obvious notation) $x \in C^{u,v}_{1,j=0}$, $y = C^{u,v}_{2,j = K_0 -1}$, and for $1 \le j < K_0$, $C^{u,v}_{2,j-1} \cap C^{u,v}_{1,j} \not= \emptyset$. Thus, if $K_0 = 1$, $p_1(x,y)$ is clearly lower bounded by the right member of (\ref{B.24}). And, if $K_0$ is bigger than $1$, one can use Chapman-Kolmogorov's identity to bound $p_1(x,y)$ from below through
\begin{equation}\label{B.25}
\begin{split}
p_1(x,y) \ge & \dis\int p_{1/K_0} ^{O_{0}}(x_0,y_1) \dots p_{1/K_0}^{O_{K_0 - 1}} (y_{K_0 - 1},y) 
\\[1ex]
& 1\{y_1 \in C_{3,0}^{u,v} \cap C^{u,v}_{1,1}, \dots , y_{K_0 - 1} \in C^{u,v}_{3, K_0 - 2} \cap C^{u,v}_{1, K_0 - 1}\} 
\\[1ex]
& d\mu^{u,v} (y_1) \dots d\mu^{u,v} (y_{K_0 - 1})
\end{split}
\end{equation}
whence (\ref{B.24}).

\medskip
To construct the above sequence of distinct charts, we note that $\IX_{u,v}$ is connected, see (\ref{A.10}), and we can find a continuous path $\pi: [0,1] \r \IX_{u,v}$ with $\pi(0) = x$ and $\pi(1) = y$. We choose a chart corresponding to $j=0$, such that $x \in C^{u,v}_{1,j = 0}$. If the path $\pi$ never leaves $C^{u,v}_{2,j=0}$, the construction stops and $K_0 = 1$. Otherwise, we consider the last time $\tau_1( < 1)$ when $\pi(\tau_1)$ belongs to $C^{u,v}_{2,j = 0}$. Then, $\pi(\tau_1)$ belongs to $C^{u,v}_{1,j=1}$ of a distinct chart, and either the path remains in $C^{u,v}_{2, j = 1}$ up to time $1$, in which case we set $K_0=2$, or there is a last time $\tau_2 (< 1)$ after which the path never belongs to $C^{u,v}_{2,j=1}$ (and $C^{u,v}_{2,j = 0}$). In this case, $\pi (\tau_2)$ belongs to $C^{u,v}_{1,j=2}$ of a chart distinct from the two preceding charts. The construction goes on, but since $K$ is the total number of charts its stops in at most $K$ steps. This completes the proof of (\ref{B.24}) (and as noted above (\ref{B.16}) follows).

\medskip
There now remains to prove (\ref{B.17}). We consider the canonical diffusion on $\IX_{u,v}$ attached to $\Gamma$ in (\ref{B.6}) with starting point $x$ in $\IX_{u,v}$, and write $P_x$ and $E_x$ for the corresponding law and expectation, as well as $(X_t)_{t \ge 0}$ for the canonical process. Then, for $y \in \IX_{u,v}$, there is a chart such that $y \in C_1^{u,v}$. We write $0 \le R_0 \le D_0 \le R_1 \le D_1 \le \dots$ for the successive times of return to $C^{u,v}_2$ and departure from $O^{u,v}$ (see (\ref{B.19}) for notation). Through the consideration of the last return to $C^{u,v}_2$ before time $1$, one has (see (\ref{B.20}) for notation)
\begin{equation}\label{B.26}
p_1(x,y) = \Sigma_{n \ge 0} \,E_x [R_n < 1, p_{1 - R_n}^{O^{u,v}} (X_{R_n},y)].
\end{equation}
(The identity is first proved in integrated form with a ``bump function'' supported near $y$. The argument below can be used to show that the right member of (\ref{B.26}) is a continuous function of $y$ and (\ref{B.26}) follows by letting the bump function concentrate at $y$.)

\medskip
Note that for $n \ge 1$, $X_{R_n}$ belongs to $\partial C_2^{u,v}$ and that $X_{R_0}$ either equals $x \in C_2^{u,v}$ or belongs to $\partial C_2^{u,v}$. As a result of (\ref{B.26}), we thus find that
\begin{equation}\label{B.27}
\begin{split}
p_1(x,y) & \le E_x [\Sigma_{n \ge 0}\, 1\{R_n < 1\}] \, \max\{p_1^{O^{u,v}}(x,y), \sup\limits_{0 < s \le 1, x^\prime \in \partial C_2^{u,v}} p_s^{O^{u,v}}(x^\prime,y)\}
\\[-2ex]
&\hspace{-2ex} \stackrel{\rm (\ref{B.22})}{\le} c^\prime \, E_x [\Sigma_{n \ge 0} \, 1\{R_n < 1\}].
\end{split}
\end{equation}
Bounding the drift and diffusion matrix of the diffusion with generator $\Gamma$, one can choose a positive integer $L_0 (u_{\min},u_{\max},\eta)$ (we refer to the beginning of the appendix and above (\ref{B.19}) for our convention about constants) such that uniformly in $z \in C_2^{u,v}$ the probability for $X$ starting at $z$ to remain in $O^{u,v}$ up to time $1/L_0$ is at least $1/2$. It then follows that for any $z \in C^{u,v}_2$, using the strong Markov property at the first return to $C^{u,v}_2$ after time $\frac{k}{L_0}$ in the second line, one has
\begin{equation}\label{B.28}
\begin{array}{l}
E_z [\Sigma_{n \ge 0} \,1\{R_n < 1\}] = \sum\limits_{0 \le k < L_0} E_z \Big[\sum\limits_{n \ge 0} 1 \Big\{ \mbox{\f $\dis\frac{k}{L_0}$} \le R_n < \mbox{\f $\dis\frac{k+1}{L_0}$}\Big\}\Big] \le
\\[2ex]
L_0 \sup\limits_{z^\prime \in C_2^{u,v}} E_{z^\prime} \Big[ \sum\limits_{n \ge 0} 1 \Big\{R_n < \mbox{\f $\dis\frac{1}{L_0}$}\Big\}\Big] \le L_0 \sum\limits_{n \ge 0} \; \mbox{\f $\dis\frac{1}{2^n}$} = 2 L_0.
\end{array}
\end{equation}
Bounding the last expectation of (\ref{B.27}) by $2L_0$ thus concludes the proof of (\ref{B.17}), and hence of Proposition \ref{propB.2}.
\end{proof}

The above Proposition \ref{propB.2} will be especially useful in Section 5, see below (\ref{5.25}) in the proof of Theorem \ref{theo5.1}. 
\end{appendix}

\end{document}